\documentclass[11pt,a4paper]{article} 
\usepackage{amsmath,amssymb,latexsym,mathrsfs,amsthm,amscd}
\input amssym.def 
\def\R{\Bbb R}    \def\Q{\Bbb Q} 
\def\I{\Bbb I} \def\N{\Bbb N}   \def\F{\Bbb F} 
  \def\Z{\Bbb Z} \def\A{\Bbb A} \def\S{\Bbb S}

\newtheorem{theo}{Theorem}[section] 
\newtheorem{lem}[theo]{Lemma} 
\newtheorem{pro}[theo]{Proposition} 
 
\newtheorem{cor}[theo]{Corollary}

\title{On Four-Dimensional Unital Division Algebras\\ over Fields of Characteristic not 2}

\author{Ernst Dieterich\\[1ex]{\it Dedicated to the memory of Peter Gabriel (1933--2015)}}

\begin{document}

\maketitle

\begin{abstract}
\noindent 
In \cite{BA15}, an exhaustive construction is achieved for the class of all 4-di\-mensional 
unital division algebras over finite fields of odd order, whose left nucleus 
is not minimal and whose automorphism group contains Klein's four-group.

We generalize the approach of \cite{BA15} towards all division algebras of the above specified type, but now admitting arbitrary fields $k$ of cha\-racteristic not 2 as ground fields. 
For these division algebras we present an exhaustive construction that depends on a quadratic field extension of $k$ and three parameters in $k$, and we derive an isomorphism criterion 
in terms of these parameters. As an application we classify, for $k_o$ an ordered field in which every positive element is a square, all division $k_o$-algebras of the mentioned type, 
and in the finite field case we refine the Main Theorem of \cite{BA15} to a classification even of the division algebras studied there. 

The category formed by the division $k$-algebras investigated here is a groupoid, whose structure we describe in a supplementary section in terms of a covering by group actions. In particular,
we exhibit the automorphism groups for all division algebras in this groupoid. 
\end{abstract}

\\[1ex]

\section{Introduction}

For any field $k$, we denote by $\mathscr{D}_4(k)$ the category of all 4-dimensional division algebras over $k$, not assumed to be associative, whose morphisms are 
the non-zero algebra morphisms. This category is in fact a groupoid, i.e.~all morphisms in $\mathscr{D}_4(k)$ are isomorphisms.

The present article is devoted to the investigation of the full subgroupoid $\mathscr{C}(k)$ of $\mathscr{D}_4(k)$, formed by all $A \in \mathscr{D}_4(k)$ which are unital, 
such that $k1$ is properly contained in the right nucleus $N_r(A)$, and whose automorphism group admits Klein's four-group as a subgroup. Our interest in the groupoid $\mathscr{C}(k)$ is twofold. 
Firstly, we would like to classify $\mathscr{C}(k)$ up to isomorphism, and secondly, we want to understand its categorical structure. 

To motivate our first interest we note that, although classification problems generally tend to provoke deepened insight, classifications of division algebras are in fact rare. Against this 
background, the recent paper \cite{BA15} stimulated our approach to a classification of $\mathscr{C}(k)$, which generalizes the approach of \cite{BA15} as follows. Assume that ${\rm char}(k) \not= 2$.
Setting out from any quadratic field extension $\ell$ of $k$ and any scalar triple $\underline{c} = (c_1,c_2,c_3) \in k^3$ that satis\-fies an implicitly defined $\ell$-admissibility condition, we construct
a division algebra $A(\ell,\underline{c}) \in \mathscr{C}(k)$. Up to isomorphism, every $A \in \mathscr{C}(k)$ can be constructed in this way. We say that two $\ell$-admissible triples $\underline{c}$ and
$\underline{d}$ are $\ell^\ast$-equivalent, if the algebras $A(\ell,\underline{c})$ and $A(\ell,\underline{d})$ are isomorphic. We find a criterion that expresses $\ell^\ast$-equivalence of $\underline{c}$
and $\underline{d}$ in terms of $\underline{c}, \underline{d}$ and $\ell^\ast$. At this stage, the classification problem of $\mathscr{C}(k)$ is reduced to the subproblems of displaying explicitly, for all
quadratic extensions $k \subset \ell$, the subsets of all $\ell$-admissible triples $C_\ell \subset k^3$, of finding a transversal $T_\ell \subset C_\ell$ for the $\ell^\ast$-equivalence classes of $C_\ell$, and 
of eliminating redundancy of type $A(\ell,\underline{c})\ \simeq\ A(\ell',\underline{d})$ with $\ell\ \not\simeq\ \ell'$. The complexity of these three subproblems depends 
heavily on the nature of the ground field $k$, and only the second subproblem appears easy in general. For two types of fields $k$, however, we do achieve complete solutions to all three subproblems
and thereby a classification of $\mathscr{C}(k)$, namely for ordered fields $k_o$ in which every positive element is a square, and for finite fields $\F_q$ of odd order $q$. In the arithmetic case $k = \Q$,
an attempt to solve the above three subproblems readily evolves into intriguing number theoretic questions. Implementing theorems of Fermat, Gauss and Legendre in pursuit of this approach, 
substantial progress towards a classification of $\mathscr{C}(\Q)$ has been achieved in \cite{Ha17}. 

The second interest stems from our observation that general groupoids often admit a non-trivial covering by groupoids that arise naturally from group actions. These so-called {\it group action groupoids}
are very concrete mathematical objects, even when the groupoid they cover originally was given in highly abstract terms. Due to this feature, we hold that the initiation of a study of {\it coverings 
of groupoids by group actions} is well motivated. In this respect, we treat the groupoid $\mathscr{C} = \mathscr{C}(k)$ as a test case. In Section 6, a covering of $\mathscr{C}$ by group actions is 
indeed established, provided that ${\rm char}(k) \not= 2$. Its intrinsic nature turns out to vary interestingly with the various types of algebra structures which $\mathscr{C}$ comprises, namely 
non-associative division algebras, non-commutative skew fields, and fields. A language, appropriate to formulate this covering result concisely, is introduced in Subsection 6.1. 

On a technical level, the covering results of Section 6 (Corollaries \ref{Kleinian covering by group actions} and \ref{refined Kleinian covering by group actions}) emerge from the results of Sections 2-4
as follows. We exhibit a partition \mbox{$C_\ell = C_\ell^0 \sqcup C_\ell^1$} such that each $C_\ell^\nu$ 
is a union of $\ell^\ast$-equivalence classes in $C_\ell$, which for their part coincide with the orbits of a suitable group action on $C_\ell^\nu$, which in turn endows $C_\ell^\nu$ with the 
structure of a group action groupoid.\vspace{0,1cm} Now, the constructions $\mathscr{F}_\ell^\nu: C_\ell^\nu \to \mathscr{C},\ \mathscr{F}_\ell^\nu(\underline{c}) = A(\ell,\underline{c}),$ give rise 
to\vspace{0,1cm} faithful and dense functors $\mathscr{F}_\ell^\nu: C_\ell^\nu \to \mathscr{C}_\ell^\nu,\ \nu \in \{0,1\},$ where $\mathscr{C}_\ell = \mathscr{C}_\ell^0 \amalg \mathscr{C}_\ell^1$\vspace{0,1cm} is 
the full subgroupoid of $\mathscr{C}$ formed by all $A \in \mathscr{C}$ admitting a filtration of subalgebras $k \subset \ell \subset N_r(A) \subset A$, such that $\ell$ is invariant under two 
distinct commuting order 2 automorphisms of $A$, with blocks $\mathscr{C}_\ell^0 = \{ A \in \mathscr{C}_\ell\ |\ |{\rm Aut}(A)| > 4\}$ and $\mathscr{C}_\ell^1 = \{ A \in \mathscr{C}_\ell\ |\ 
|{\rm Aut}(A)| = 4\}$. As $\ell$ ranges through all isoclasses of quadratic extensions of $k$ and $\nu$ ranges through $\{0,1\}$, the subgroupoids $\mathscr{C}_\ell^\nu$ cover $\mathscr{C}$. If we ignore the 
non-commutative skew fields in $\mathscr{C}$, then the faithful and dense functors $\mathscr{F}_\ell^\nu: C_\ell^\nu \to \mathscr{C}_\ell^\nu$ are even full, i.e.~equivalences of categories.

\section{Preliminaries}

\subsection{Conventions, notation, terminology and basic notions}

In this article, the least natural number is $0$. For all $n \in \N$, we set
\[ \underline{n} = \{ i \in \N\ |\ 1 \le i \le n \}. \] 
If $F$ is a field and $(m,n) \in \N^2$, then $F^{m \times n}$ denotes the set of all \mbox{$m \times n$}-matrices with entries in $F$, and 
$F^\ast = (F \setminus \{0\}, \cdot)$ is the multiplicative group \mbox{of $F$.} The {\it square map} $s_F: F \to F,\ s_F(x) = x^2$, 
induces a group endomorphism $s_F^\ast: F^\ast \to F^\ast$. The notation $X_{sq} = s_F(X) = \{ x^2\ |\ x \in X \}$ will be used for all
subsets $X \subset F$.

The cardinality of a set $Y$ is denoted by $|Y|$. Maps are written on the left of their arguments, and are 
composed on the left. The composite map $gf$ will also be denoted by $g \circ f$, if it is beneficial to clarity.

If $(Y_i)_{i \in I}$ is a family of subclasses of a class $Z$, then $Y = \bigsqcup_{i \in I} Y_i$ expresses that $Y = \bigcup_{i \in I} Y_i$ and $Y_i \cap Y_j = \emptyset$ for all $i \not= j$.
The object class of a category $\mathscr{A}$ is also denoted by 
$\mathscr{A}$, for simplicity. A family $(\mathscr{A}_i)_{i \in I}$ of 
full subcategories of a category $\mathscr{A}$ will be called a {\it covering} of 
$\mathscr{A}$ if the identity of object classes $\mathscr{A} = \bigcup_{i \in I} 
\mathscr{A}_i$ holds true, and will be called a {\it decomposition} of $\mathscr{A}$ with 
{\it blocks} $\mathscr{A}_i$ if the identity of object classes $\mathscr{A} = \bigsqcup_{i 
\in I} \mathscr{A}_i$ holds true and ${\rm Mor}_\mathscr{A}(A_i, A_j) = \emptyset$ for all 
$(A_i, A_j) \in \mathscr{A}_i \times  \mathscr{A}_j$ and all $i \not= j$. The 
decomposition of $\mathscr{A}$ into blocks $\mathscr{A}_i$ is also expressed 
by saying that $\mathscr{A}$ is the {\it coproduct} of $(\mathscr{A}_i)_{i \in I}$ with 
{\it cofactors} $\mathscr{A}_i$, or by the formula $\mathscr{A} = \coprod_{i \in I} 
\mathscr{A}_i$. 

Following \cite{GR92}, we say that a category $\mathscr{A}$ is {\it svelte}, in case its isoclasses form a set, denoted by 
$\mathscr{A}/\hspace{-0,1cm}\simeq$. In this 
article, the term {\it classification} always means classification up to isomorphism. More precisely, by a classification 
of a svelte category $\mathscr{A}$ we mean an explicitly displayed transversal $\mathscr{T}$ for 
$\mathscr{A}/\hspace{-0,1cm}\simeq$. Synonymously, we say that $\mathscr{T}$ classifies $\mathscr{A}$. 

For any group $G$ with identity element $e$, we call $G^\circ = G \setminus \{e\}$ a {\it punctured group}. In writing $H < G$, we mean that $H$ is a subgroup of $G$. If $k$ is a field and $h < k^\ast$, 
then $\left( k^\ast/h \right)^\circ = (k^\ast/h) \setminus \{h\}$ is the {\it punctured quotient group} of $k^\ast$ by $h$. By {\it Klein's four-group} we mean any non-cyclic group of \mbox{order 4.} The 
notation ${\rm G} = {\rm C}_2 \times {\rm C}_2$ with ${\rm C}_2 = \{\pm 1\}$ for a specific multiplicative instance of Klein's four-group, and ${\rm V} = \Z_2 \times \Z_2$ with 
$\Z_2 = \{\overline{0},\overline{1}\}$ for  a specific additive instance of Klein's four-group, will be kept throughout this paper.

The symbol $k$ always denotes a field. By a $k$-{\it algebra} $A$ we mean a 
vector space $A$ over $k$, equipped with an algebra structure, i.e.~a $k$-bilinear 
map $A \times A \to A,\ (x,y) \mapsto xy$. The product $xy$ will occasionally be 
denoted by $x \cdot y$, to enhance comprehensibility. With every 
\mbox{$k$-algebra} $A$ and every element $a \in A$, we associate the $k$-linear operators 
\mbox{$L_a: A \to A,$}\ $L_a(x) = ax$ and $R_a: A \to A,\ R_a(x) = xa$. A {\it division 
algebra} over $k$ is a non-zero $k$-algebra $A$ such that both $L_a$ and $R_a$ are 
bijective for all $a \in A \setminus \{0\}$. A $k$-algebra $A$ is called {\it 
unital} if it has an {\it unity} $1 = 1_A$, such that $1x = x = x1$ for all $x 
\in A$. For unital $k$-algebras $A$, we usually treat the canonical field isomorphism 
$k\ \tilde{\to}\ k1$ as identity $k = k1$, to ease notation.

A {\it Hurwitz algebra} over $k$ is a non-zero unital $k$-algebra $A$, admitting a 
non-degenerate multiplicative quadratic form $q: A \to k$, traditionally called the 
{\it norm} of $A$. 

A {\it skew field} over $k$ is a unital associative division algebra over $k$. With each skew field $A$ over $k$ we associate its {\it multiplicative group} $A^\ast = (A \setminus \{0\},\cdot)$. 
A {\it skew field} over $k$ is called {\it central} if its centre equals $k$. The commutative skew fields $\ell$ over $k$ are just the field extensions $k \subset \ell$. Their dimension over $k$ is 
traditionally called {\it degree} of the field extension, and is denoted by $\dim_k(\ell)$ or $[\ell:k]$. Field extensions of degree 2 are briefly referred to as {\it quadratic extensions}. 

A {\it morphism} from a $k$-algebra $A$ to a $k$-algebra $B$ is a $k$-linear map 
$\varphi: A \to B$ such that $\varphi(xy) = \varphi(x)\varphi(y)$ for all $x,y \in A$. An 
isomorphism of $k$-algebras is a bijective morphism. Thus, the automorphism group 
of a $k$-algebra $A$ is 
\[ {\rm Aut}(A) = \{ \varphi \in {\rm GL}_k(A)\ |\ \varphi(xy) = \varphi(x)\varphi(y)\ 
\forall x,y \in A \}. \] 
Its identity element will be denoted by $\I = \I_A$. If $A$ is unital, then every automorphism $\varphi$ of $A$ fixes 
$k = k1$ elementwise. In particular, if a $k$-algebra $\ell$ is a field extension $k \subset \ell$, then ${\rm Aut}(\ell)$ coincides 
with the group of all $k$-automorphisms of the field $\ell$, also known as the Galois group of $\ell$ over $k$, and denoted 
${\rm Gal}(\ell/k)$. Furthermore, if $A$ is a skew field over $k$, then 
every $a \in A^\ast$ determines an inner automorphism $\kappa_a \in {\rm Aut}(A)$, 
defined by $\kappa_a(x) = axa^{-1}$. A famous Theorem of Skolem and Noether asserts that, if $A$ 
is a central skew field over $k$, then ${\rm Aut}(A) = \{ \kappa_a\ |\ a \in A^\ast \}$. In 
that case, the surjective group morphism $\kappa: A^\ast \to {\rm Aut}(A),\ \kappa(a) = 
\kappa_a$ induces the isomorphism $\overline{\kappa}: A^\ast/k^\ast\ \tilde{\to}\ 
{\rm Aut}(A),\ \overline{\kappa}(\overline{a}) = \kappa_a$. 

A morphism from a {\it division} $k$-algebra $A$ to a {\it division} $k$-algebra $B$ is understood to be a {\it non-zero} algebra morphism $\varphi: A \to B$. With this convention, every
morphism of division algebras is injective, and every morphism of division algebras of equal finite dimension is an isomorphism \cite[Proposition 2.3]{Di12}.

By a {\it groupoid} we mean a category, in which all morphisms are isomorphisms. Thus, for every field $k$ and all $n \in \N$, the category $\mathscr{D}_n(k)$ of all $n$-dimensional 
division algebras over $k$ is a groupoid. (It may or may not be empty.) Consequently, every full subcategory of $\mathscr{D}_n(k)$ is a groupoid. In particular, the full subcategory 
$\mathscr{C} = \mathscr{C}(k)$ of $\mathscr{D}_4(k)$, introduced in \mbox{Section 1,} is a groupoid. It contains the full subgroupoids $\mathscr{N} = \mathscr{N}(k),\ \mathscr{S} = \mathscr{S}(k)$ 
and $\mathscr{K} = \mathscr{K}(k)$ formed by
\[ \begin{array}{ccl} 
\mathscr{N} & = & \{ A \in \mathscr{C}\ |\ A\ \mbox{is\ not associative}\},
\\[1ex]
\mathscr{S} & = & \{ A \in \mathscr{C}\ |\ A\ \mbox{is\ a\ central\ skew\ field\ over}\ k\}\ \mbox{and}
\\[1ex]
\mathscr{K} & = & \{ A \in \mathscr{C}\ |\ A\ \mbox{is\ a\ field\ extension\ of}\ k\}.
\end{array} \] 
In the same vein, $\mathscr{D}_2(k)$ contains the full subgroupoid $\mathscr{Q} = \mathscr{Q}(k)$ formed by all quadratic extensions of $k$.

The groupoids $\mathscr{N}, \mathscr{S}, \mathscr{K}$ and $\mathscr{Q}$ will play a prominent role in the investigation of $\mathscr{C}$, throughout this article.

\subsection{Generalities on quadratic extensions}

Let $k \subset \ell$ be a quadratic extension, and ${\rm char}(k) \not= 2$. Then $k \subset \ell$ is {\it Galois}, with Galois group ${\rm Gal}(\ell/k) = \langle \sigma \rangle$ of order 2. 
We write $\overline{x} = \sigma(x)$ for all $x \in \ell$. The set of all {\it purely imaginary elements} in $\ell$,
\[ {\rm Im}(\ell) = \{ x \in \ell\ |\ \overline{x} = -x \} = \{ x \in \ell\ |\ x^2 \in k \} \setminus k^\ast, \]
is a 1-dimensional $k$-linear subspace in $\ell$, such that $\ell = k \oplus {\rm Im}(\ell)$. 

The {\it norm map} $n_{\ell/k}: \ell \to k,\ n_{\ell/k}(x) = x\overline{x}$, induces a group morphism 
$n_{\ell/k}^\ast: \ell^\ast \to k^\ast,\ n_{\ell/k}^\ast(x) = x\overline{x}$. Besides, we have the group endomorphism 
$s_\ell^\ast: \ell^\ast \to \ell^\ast,\ s_\ell^\ast(x) = x^2$, introduced in Subsection 2.1.

For any subgroup $H < \ell^\ast$ with $\sigma(H) = H$, the symbol 
$H >\hspace{-0,2cm}\lhd\ {\rm Gal}(\ell/k)$ denotes the {\it semi-direct product} of $H$ 
and ${\rm Gal}(\ell/k)$, i.e.~the group with underlying set $H \times {\rm Gal}(\ell/k)$ 
and binary operation $(a, \sigma^i)(b, \sigma^j) = (a\sigma^i(b), \sigma^{i+j})$. The 
semi-direct products \mbox{$H >\hspace{-0,2cm}\lhd\ {\rm Gal}(\ell/k)$} that will arise naturally in our context involve three subgroups $H_i < \ell^\ast$, 
namely $(H_1, H_2, H_3) = (\ell^\ast, \S, \A^\ast)$, where the ``unit sphere'' $\S$ and the ``punctured axes'' $\A^\ast$ are defined by
\[ \begin{array}{lclcl}
\S\hspace{0,2cm} =\ \S(\ell/k) &=& \{ x \in \ell^\ast\ |\ x\overline{x} = 1 \} &=& {\rm ker}(n_{\ell/k}^\ast),\hspace{0,1cm}\ \mbox{and}
\\[0,2cm] 
\A^\ast = \A^\ast(\ell/k) &=& \{x \in \ell^\ast\ |\ x^2 \in k^\ast \} &=& (s_\ell^\ast)^{-1}(k^\ast) = k^\ast \sqcup ({\rm Im}(\ell) \setminus \{0\}).  
\end{array} \]

\subsection{Decomposition of $\mathscr{C}(k)$}

Let $k$ be any field. We show that every division algebra $A \in \mathscr{C}(k)$ either is non-associative, or a central skew field over $k$,
or a field extension of $k$, and we supplement this decomposition statement on $\mathscr{C}(k)$ by equivalent characterizations of each block. 
To begin with, we recall some basic notions needed in this context, and notation that goes along with them. 

Let $A$ be any $k$-algebra. A {\it subalgebra} of $A$ is a $k$-linear subspace $U 
\subset A$, such that $xy \in U$ for all $x,y \in U$. Every subalgebra $U$ of $A$ is itself a $k$-algebra. 
A {\it subfield} of $A$ is a subalgebra $U \subset A$, such that $U$ is a field. The {\it right nucleus} of $A$ is,
by definition, the subset 
\[ N = N_r(A) = \{ z \in A\ |\ (xy)z = x(yz)\ \forall x,y \in A \} \]
of $A$. Every division algebra $A$ over $k$ has {\it no zero divisors}, i.e. $xy = 0$ holds in $A$ 
only if $x = 0$ or $y = 0$. If, conversely, $A$ has no zero divisors and is finite 
dimensional, then $A$ is a division algebra over $k$. It follows, that every 
subalgebra of any finite-dimensional division algebra over $k$ again is a 
finite-dimensional division algebra over $k$.

The full subgroupoids $\mathscr{C}(k), \mathscr{N}(k), \mathscr{S}(k), \mathscr{K}(k)$ of $\mathscr{D}_4(k)$ are defined in Section 1 and Subsection 2.1. Besides, we introduce the full subgroupoid
$\mathscr{HD}_4(k)$ of $\mathscr{D}_4(k)$, formed by all 4-dimensional Hurwitz division algebras over $k$.   

\begin{lem} \label{nucleus}
Let $A$ be an algebra over a field $k$.
\\[1ex]
(i) The right nucleus $N$ of $A$ is an associative subalgebra of $A$.
\\[1ex]
(ii) If $A$ is unital, then $k \subset N \subset A$ is a filtration of subalgebras, and $N$ is a unital associative subalgebra of $A$. 
Moreover, $A$ is a right $N$-module, whose module structure $A \times N \to A$ is the restricted algebra structure of $A$.
\\[1ex]
(iii) If $A$ is a finite-dimensional unital division algebra over $k$, then $N$ is a skew 
field over $k$, and $A$ is a right vector space over $N$, such that
\[ \dim_k(A) = \dim_N(A) \dim_k(N). \]
If in addition $\dim_k(N) = 2$, then $N$ is a field.
\end{lem}

\begin{proof}
$(i)$ and $(ii)$ admit straightforward verifications.
\\[1ex]
$(iii)$ By (ii), $N$ is a unital associative subalgebra of the finite-dimensional division algebra $A$ over $k$. Hence $N$ is a unital 
associative division algebra over $k$, i.e.~a skew field over $k$. Consequently, again by (ii), $A$ is a right vector space over $N$.
If $(x_i)_{i \in I}$ is an $N$-basis in $A$ and $(y_j)_{j \in J}$ is a 
$k$-basis in $N$, then $(x_i y_j)_{ij \in I \times J}$ is a $k$-basis in $A$. This proves the 
dimension formula. If $\dim_k(N) = 2$, then $N$ admits a $k$-basis $(1,b)$, whence we conclude that $N$ is a commutative skew field, i.e.~a field. 
\end{proof}

\begin{pro} \label{decomposition}
For every field $k$, the decomposition  
\[ \mathscr{C}(k) = \mathscr{N}(k) \amalg \mathscr{S}(k) \amalg \mathscr{K}(k) \]
holds true. Moreover, $\mathscr{N}(k) = \{ A \in \mathscr{C}(k)\ |\ \dim_k(N) = 2 \}$, and $\mathscr{K}(k)$ is formed by all Galois 
extensions of $k$ whose Galois group is Klein's four-group. If ${\rm char}(k) \not= 2$, then $\mathscr{S}(k) \supset \mathscr{HD}_4(k)$.
\end{pro}

\begin{proof}
The object classes $\mathscr{N}(k), \mathscr{S}(k)$ and $\mathscr{K}(k)$ are pairwise disjoint, and closed under isomorphisms within the groupoid $\mathscr{C}(k)$. Therefore, in order to prove 
the asserted decomposition of $\mathscr{C}(k)$, it suffices to show that $\mathscr{C}(k) \subset \mathscr{N}(k) \sqcup \mathscr{S}(k) \sqcup \mathscr{K}(k)$. So, let $A \in \mathscr{C}(k)$ be given.
If $A$ is not associative, then $A \in \mathscr{N}(k)$. If $A$ is associative, then $A$ is a skew field over $k$. Let $Z$ be the centre of $A$, and let $E$ be a maximal subfield of $A$. Then 
\cite[Theorem 7.15 (i)]{Re75} implies that  $k \subset Z \subset E \subset A$ and 
\[ \dim_k(A) = \dim_Z(A) \dim_k(Z) = \dim_Z(E)^2 \dim_k(Z). \]  
Thus $\dim_Z(A)$ divides 4 and is a square, whence $\dim_Z(A) \in \{ 1, 4 \}$. If $\dim_Z(A) = 4$, then $\dim_k(Z) = 1$ shows that $A$ is a central skew field \mbox{over $k$,} i.e.~$A \in \mathscr{S}(k)$.
If $\dim_Z(A) = 1$, then $\dim_k(Z) = 4$ shows that $A$ is a commu\-tative skew field over $k$, i.e.~$A \in \mathscr{K}(k)$.
\\[1ex]
For all $A \in \mathscr{C}(k)$, the dimension formula in Lemma \ref{nucleus} (iii) implies that $\dim_k(N) \in \{ 2, 4 \}$. Since $A$ is associative if and only if $A = N$, we conclude that $A$ is 
not associative if and only if $\dim_k(N) = 2$.
\\[1ex]
Let $A \in \mathscr{K}(k)$. Then Klein's four-group appears as a subgroup $H < {\rm Aut}(A)$. By \cite[Propositions V.3.7 and V.3.6]{Gr07}, the fixed field 
\[ F = A^H = \{ x \in A\ |\ \sigma(x) = x\ \forall \sigma \in H \} \] 
is an intermediate field $k \subset F \subset A$ such that $F \subset A$ is a Galois extension with Galois group $H$, whose degree is 
$[A:F] = |{\rm Gal}(A/F)| = |H| = 4$. It follows that $[F:k] = 1$, i.e.~$F = k$. Conversely, if $k \subset A$ is a Galois extension whose Galois group is Klein's four-group, then 
$[A:k] = |{\rm Gal}(A/k)| = 4$ and ${\rm Gal}(A/k) = {\rm Aut}(A)$, so $A \in \mathscr{K}(k)$.     
\\[1ex]
Classical theory of Hurwitz algebras states that every $A \in \mathscr{HD}_4(k)$ is a central skew field over $k$ \cite[Theorem 1.6.2 and Proposition 1.9.1]{SV00}. Therefore, in order to show that
$A \in \mathscr{S}(k)$, it suffices to exhibit Klein's four-group as a subgroup of ${\rm Aut}(A)$. For that purpose, we make use of the Skolem-Noether isomorphism 
$\overline{\kappa}: A^\ast/k^\ast\ \tilde{\to}\ {\rm Aut}(A),\ \overline{\kappa}(\overline{a}) = \kappa_a$, explained in Subsection 2.1. If ${\rm char}(k) \not= 2$, then \cite[Corollary 1.6.3]{SV00} 
asserts the existence of an orthogonal $k$-basis $(1, a, b, ab)$ in $A$. It follows that $\overline{a} \not= \overline{b}$ and, invoking
\cite[Proposition 1.2.3]{SV00}, that $\overline{a}^2 = \overline{b}^2 = \overline{1}$ and $\overline{a}\overline{b} = 
\overline{b}\overline{a}$. Accordingly, the subgroup $\langle \kappa_a, \kappa_b \rangle < {\rm Aut}(A)$ is Klein's four-group.
\end{proof}

\noindent
The inclusion of object classes $\mathscr{S}(k) \supset \mathscr{HD}_4(k)$, which we just proved, is in fact an identity (Corollary \ref{Hurwitz division algebras}). Our proof of the converse inclusion 
$\mathscr{S}(k) \subset \mathscr{HD}_4(k)$ must be postponed, as it is based on a closer analysis of the groupoid $\mathscr{C}(k)$, to be developed in Sections 3 and 4.

\section{Constructions}

Throughout this section, $k$ is a field of characteristic not 2, and $k \subset \ell$ is a quadratic extension. We refer to Subsection 2.2 for notation that goes along with these data.
In Subsection 3.1 we define the subset $C_\ell \subset k^3$, formed by all $\ell$-admissible triples in $k^3$. In Subsection 3.2 we construct from any \mbox{$\ell$-admissible} triple $\underline{c} \in C_\ell$ 
a division algebra $A(\ell,\underline{c}) \in \mathscr{C}(k)$, and we study how this 
construction behaves with respect to the decomposition of $\mathscr{C}(k)$. In Subsection 3.3 we associate with any pair $(\underline{c},\underline{d}) \in C_\ell \times C_\ell$ a subset 
$\ell^\ast(\underline{c},\underline{d}) \subset \ell^\ast$, and we construct from each $a \in \ell^\ast(\underline{c},\underline{d})$ two morphisms $\varphi_a: A(\ell,\underline{c}) \to A(\ell,\underline{d})$ 
and $\psi_a: A(\ell,\underline{c}) \to A(\ell,\underline{d})$ in $\mathscr{C}(k)$.

\subsection{The notion of an $\ell$-admissible triple}

A general function $f: \ell^2 \to \ell$ with $f(0,0) = 0$ is said to be {\it anisotropic} if $f^{-1}(0) = \{(0,0) \}$, and 
{\it isotropic} otherwise. With any triple of scalars 
\[ \underline{c} = (c_1, c_2, c_3) \in k^3 \] 
we associate a function
\[ q_{\underline{c}}: \ell^2 \to \ell,\ q_{\underline{c}}(x,y) = (1-c_1)x^2 + c_1 x\overline{x} - c_2 y^2 - c_3 y\overline{y}. \]
We say that a triple $\underline{c} \in k^3$ is $\ell$-{\it admissible} if the function $q_{\underline{c}}: \ell^2 \to \ell$ is anisotropic. The $\ell$-admissible triples in $k^3$ form a subset 
\[ C_\ell = \left\{ \underline{c} \in k^3\ \left|\ q_{\underline{c}}^{-1}(0) = \{(0,0)\} \right. \right\} \]
of $k^3$. Depending on the quadratic extension $k \subset \ell$, the subset $C_\ell \subset k^3$ may or may not be empty.

\subsection{Construction of objects in $\mathscr{C}(k)$}

From any triple $\underline{c} \in k^3$ we construct a 4-dimensional $k$-algebra $A(\ell,\underline{c})$, de\-fining it as the vector space $\ell^{2 \times 1}$ over $k$, equipped with the algebra structure
\[ \left( \begin{array}{c} x\\y \end{array} \right) \cdot \left( \begin{array}{c} w\\z \end{array} \right) = \left( \begin{array}{cc} x & c_2y + c_3\overline{y} \\ y & (1-c_1)x + c_1\overline{x} 
\end{array} \right) \left( \begin{array}{c} w\\z \end{array} \right), \]
where the right hand side is a product of $\ell$-matrices. 

\begin{pro} \label{constructed algebras}
Let $k \subset \ell$ be a quadratic extension in characteristic \mbox{not 2,} and let $\underline{c} \in k^3$. Then the 4-dimensional $k$-algebra $A = A(\ell,\underline{c})$ has the following properties.
\\[1ex]
(i) $A$ is unital, with unity
\[ 1_A = \left( \begin{array}{c} 1\\0 \end{array} \right). \] 
(ii) The subspace 
\[ 1_A \ell = \left( \begin{array}{c} \ell \\ 0 \end{array} \right) \subset A \] 
is a $k$-subalgebra, canonically isomorphic to $\ell$, and such that 
\[ 1_Ak \subset 1_A\ell \subset N_r(A) \subset A \]
is a filtration of $k$-subalgebras.
\\[1ex]
(iii) The operators $\alpha$ and $\beta$ on $A$, defined by 
\[ \alpha \left( \begin{array}{c} x\\y \end{array} \right) = \left( \begin{array}{r} x\\-y \end{array} \right)\hspace{1ex} \mbox{and}\hspace{1ex}
\beta \left( \begin{array}{c} x\\y \end{array} \right) = \left( \begin{array}{c} \overline{x} \\ \overline{y} \end{array} \right),  \]
are algebra automorphisms of $A$, which generate Klein's four-group.
\\[1ex]
(iv) $A(\ell,\underline{c})$ is a division algebra if and only if $\underline{c}$ is $\ell$-admissible.
\\[1ex]
\end{pro}

\begin{proof}
{\it (i)-(ii)} One verifies directly that $1_A = \left( \begin{array}{c} 1\\0 \end{array} \right)$ is a unity in $A$,\vspace{0,1cm} and that $1_A \ell$ is a subalgebra of $A$, which is 
canonically isomorphic to $\ell$. By construction, $A$ is also a 2-dimensional right vector space over $\ell$, such that the identity
\begin{equation} a \cdot (bl) = (a \cdot b)l \end{equation} 
holds for all $(a,b,l) \in A \times A \times \ell$. Setting $b = 1_A$, we see that
\begin{equation} a \cdot 1_Al = al \end{equation} 
holds for all $(a,l) \in A \times \ell$. Now (1) and (2) imply $1_A \ell \subset N_r(A)$, which establishes the claimed filtration of $k$-subalgebras.
\\[1ex]
{\it (iii)} One verifies directly, that $\alpha$ and $\beta$ are distinct commuting automorphisms of $A$, both of order 2. 
\\[1ex]
{\it (iv)} 
For all $a \in A$, the $k$-linear operator $L_a: A \to A$ is even $\ell$-linear, by (1). 
If $a = \left( \begin{array}{c} x\\y \end{array} \right)$, then
\[ \det(L_a) = \det \left( \begin{array}{cc} x & c_2y + c_3\overline{y} \\ y & (1-c_1)x + c_1\overline{x} \end{array} \right) = q_{\underline{c}}(x,y). \]
Now, $A$ is a division algebra if and only if $\det(L_a) \not= 0$ for all $a \in A \setminus \{0\}$. Equivalently, the function $q_{\underline{c}}: \ell^2 \to \ell$ is anisotropic, i.e.~$\underline{c}$ is 
$\ell$-admissible. 
\end{proof}

\noindent
As a direct consequence of Proposition \ref{constructed algebras}, we obtain the following construction of objects in $\mathscr{C}(k)$.

\begin{cor} \label{constructed division algebras}
If $\underline{c} \in C_\ell$, then $A(\ell,\underline{c}) \in \mathscr{C}(k)$.
\end{cor}

\noindent
Let us summarize and introduce new notation. For every quadratic extension $k \subset \ell$ in characteristic not 2, we constructed a map 
\[ \mathscr{F}_\ell: C_\ell \to \mathscr{C},\ \mathscr{F}_\ell(\underline{c}) = A(\ell,\underline{c}), \]
where $\mathscr{C} = \mathscr{C}(k)$. The decomposition $\mathscr{C} = \mathscr{N} \amalg \mathscr{S} \amalg \mathscr{K}$ (Proposition \ref{decomposition}) induces via $\mathscr{F}_\ell$ a partition 
$C_\ell = C_\ell^{\mathscr{N}} \sqcup C_\ell^{\mathscr{S}} \sqcup C_\ell^{\mathscr{K}}$, whose subsets are implicitly\vspace{0,1cm} defined by $C_\ell^{\mathscr{N}} = \mathscr{F}_\ell^{-1}(\mathscr{N}),\ 
C_\ell^{\mathscr{S}} = \mathscr{F}_\ell^{-1}(\mathscr{S})$, and $C_\ell^{\mathscr{K}} = \mathscr{F}_\ell^{-1}(\mathscr{K})$. For later applications, we proceed to display these subsets explicitly.

\begin{pro} \label{types of admissible triples}
For all $\underline{c} \in C_\ell$, the following holds true.
\\[1ex]
(i) $A(\ell,\underline{c}) \in \mathscr{N}$ if and only if $(c_1,c_2) \not= (1,0)$ and $(c_1,c_3) \not= (0,0)$.
\\[1ex]
(ii) $A(\ell,\underline{c}) \in \mathscr{S}$ if and only if $(c_1,c_2) = (1,0)$.
\\[1ex]
(iii) $A(\ell,\underline{c}) \in \mathscr{K}$ if and only if $(c_1,c_3) = (0,0)$.
\end{pro}

\begin{proof}
Given $\underline{c} \in C_\ell$, set $A = A(\ell,\underline{c}),\ N = N_r(A),\ j = \left( \begin{array}{c} 0\\1 \end{array} \right) \in A$,\vspace{0,1cm} and choose $i \in {\rm Im}(\ell) \setminus \{0\}$. 
We will repeatedly use Proposition \ref{decomposition}.
\\[1ex]
{\it (i)} Assume that $(c_1, c_2) = (1,0)$ or $(c_1, c_3) = (0, 0)$. Then one verifies directly that $j \in N$. Since $A = 1_A\ell \oplus j\ell = 1_A\ell \oplus j \cdot (1_A \ell)$ by (2), and 
$1_A \ell \subset N$ by Proposition \ref{constructed algebras} (ii), we conclude that $A = N$. Thus $A$ is associative, i.e.~$A \not\in \mathscr{N}$. Assume conversely that $A \not\in \mathscr{N}$, 
i.e.~$A$ is associative. Then direct calculation shows that the system
\[ \left\{ \begin{array}{ccc} 
(j \cdot ji) \cdot j & = & j \cdot (ji \cdot j) \\
(ji \cdot j) \cdot j & = & ji \cdot (j \cdot j) 
\end{array} \right. \]
is equivalent to
\[ \left\{ \begin{array}{rcl}  -c_1c_2 + (1-c_1)c_3 & = & 0 \\ 
c_1c_2 + (1-c_1)c_3 & = & 0,
 \end{array} \right. \]
which in turn is equivalent to
\[ \left\{ \begin{array}{rcl}  
c_1c_2 & = & 0  \\ 
(1-c_1)c_3 & = & 0. 
\end{array} \right. \]
Since $1 \not= 0$ excludes $c_1 = 1-c_1 = 0$, and $q_{\underline{c}}(0,1) \not= 0$ excludes 
$c_2 = c_3 = 0$, we are left with the two alternatives $(c_1, c_2) = (1,0)$ or 
$(c_1, c_3) = (0, 0)$.
\\[1ex]
{\it (ii)-(iii)} If $(c_1, c_2) = (1,0)$, then $A \in \mathscr{S} \amalg \mathscr{K}$ holds by (i). Moreover, $q_{\underline{c}}(0,1) \not= 0$ implies $c_3 \not= 0$, which in turn implies 
$j \cdot ji \not= ji \cdot i$. So $A$ is not commutative, and hence $A \in \mathscr{S}$.

If $(c_1, c_3) = (0, 0)$, then one verifies directly that $A$ is commutative. Together with (i), it follows that $A \in \mathscr{K}$. 
 
In view of (i) and $\mathscr{S} \cap \mathscr{K} = \emptyset$, the ``only if'' statements in (ii) and (iii) are now direct consequences.
\end{proof}

\subsection{Construction of morphisms in $\mathscr{C}(k)$}

Every pair $(\underline{c}, \underline{d}) \in C_\ell \times C_\ell$ determines a subset $\ell^\ast(\underline{c}, \underline{d}) \subset \ell^\ast$, defined by  
\[ \ell^\ast(\underline{c}, \underline{d}) = \left\{ a \in \ell^\ast\ \left|\ \left(c_1, \frac{c_2}{a^2}, \frac{c_3}{a\overline{a}} \right) = (d_1, d_2, d_3) \right. \right\}. \]
Direct verification establishes the following lemma.

\begin{lem} \label{constructed morphisms}
For all $\underline{c}, \underline{d} \in C_\ell$ and $a \in \ell^\ast$, the following holds true.
\\[1ex]
(i) $a \in \ell^\ast(\underline{c}, \underline{d})$ if and only if $\overline{a}\in \ell^\ast(\underline{c}, \underline{d})$. 
\\[1ex]
(ii) If $a \in \ell^\ast(\underline{c}, \underline{d})$, then the maps 
\[ \varphi_a: A(\ell,\underline{c}) \to A(\ell,\underline{d}),\ \varphi_a \left( \begin{array}{c} x\\y \end{array} \right) = \left( \begin{array}{r} x\\ay \end{array} \right) \]
and
\[ \psi_a: A(\ell,\underline{c}) \to A(\ell,\underline{d}),\ \psi_a \left( \begin{array}{c} x\\y \end{array} \right) = \left( \begin{array}{r} \overline{x}\\a\overline{y} \end{array} \right) \]
are morphisms in $\mathscr{C}(k)$. 
\\[1ex]
(iii) $\{ \pm 1 \} \subset \ell^\ast(\underline{c},\underline{c})$, and the automorphisms $\varphi_{-1}$ and $\psi_1$ of $A(\ell,\underline{c})$ coincide with the automorphisms $\alpha$ and $\beta$  of 
$A(\ell,\underline{c})$, displayed in \mbox{Proposition \ref{constructed algebras} (iii).}
\end{lem}

\section{Reductions}

Even in this section, $k$ is a field of characteristic not 2. Tracing the constructions of the previous section in the reverse direction, we achieve reductions of objects and 
morphisms in $\mathscr{C}(k)$, that can be phrased as follows. For every division algebra $A \in \mathscr{C}(k)$ there is a quadratic extension $k \subset \ell$ and an $\ell$-admissible triple 
$\underline{c} \in C_\ell$ such that $A(\ell,\underline{c})\ \tilde{\to}\ A$. For every quadratic extension $k \subset \ell$ and all $(\underline{c}, \underline{d}) \in C_\ell \times C_\ell$, the division
algebras $A(\ell,\underline{c})$ and $A(\ell,\underline{d})$ are isomorphic if and only if the subset $\ell^\ast(\underline{c},\underline{d}) \subset \ell^\ast$ is not empty. 

These reduction statements can be proved in the way of a generalization of \cite{BA15} from finite fields of odd order to general fields $k$ of characteristic not 2. This proof is 
elementary, but yet laborious and somewhat obscure. Instead, we present here a conceptual and more lucid approach, suggested by an anonymous referee, to whom credit must be given for this improvement. 
It is based on the insight, that every $A \in \mathscr{C}(k)$ is {\rm V}-graded, and on the usefulness of the trace bilinear form of $A$.

\subsection{{\rm V}-grading and trace bilinear form of objects in $\hat{\mathscr{C}}(k)$}

If $\varphi$ is an automorphism of a $k$-algebra $A$ and $\varepsilon \in k$ is an eigenvalue of $\varphi$, then we denote by $E_\varphi(\varepsilon) = \{ x \in A\ |\ \varphi(x) = \varepsilon x \}$ the 
eigenspace of $\varphi$ associated with the eigenvalue $\varepsilon$.

Let $\alpha$ and $\beta$ be automorphisms of a $k$-algebra $A$. We say that $(\alpha,\beta)$ is a {\it Kleinian pair} for $A$ if $\langle \alpha,\beta \rangle$ is Klein's four-group. Equivalently,
$\alpha$ and $\beta$ are distinct commuting automorphisms of order 2.

Because the right nucleus of a division algebra $A \in \mathscr{C}(k)$ is irrelevant for the material to be expounded in this subsection, we introduce the full subgroupoid $\hat{\mathscr{C}}(k)$ of 
$\mathscr{D}_4(k)$, formed by all $A \in \mathscr{D}_4(k)$ which are unital and admit a Kleinian pair. With this definition, $\mathscr{C}(k) \subset \hat{\mathscr{C}}(k) \subset \mathscr{D}_4(k)$ is a
filtration of full subgroupoids.

Recall our notation ${\rm G} = {\rm C}_2 \times {\rm C}_2$ and ${\rm V} = \Z_2 \times \Z_2$, for two specific instances of Klein's four-group. The identity element in ${\rm G}$ is $e = (1,1)$, 
and general elements in ${\rm V}$ will be denoted by $ij = (\overline{i},\overline{j})$, for brevity.

\begin{pro} \label{V-grading}
For every division algebra $A \in \hat{\mathscr{C}}(k)$ with Kleinian pair $(\alpha,\beta)$, the subspaces $A_{ij} = A_{ij}^{\alpha\beta} = E_\alpha \left((-1)^i \right) \cap E_\beta \left( (-1)^j \right)$
of $A$,\vspace{0,1cm} $ij \in {\rm V}$, form a {\rm V}-grading $A = \bigoplus_{ij \in {\rm V}} A_{ij}$ of $A$, such that \mbox{$\dim_k(A_{ij}) = 1$} for all $ij \in {\rm V}$, and $A_{00} = k1_A$. 
\end{pro}

\begin{proof}
The group algebra $k{\rm G}$ is semisimple, because ${\rm char}(k)\hspace{-0,1cm} \not|\ |{\rm G}|$. It admits four isoclasses of simple modules, represented by the 1-dimensional $k{\rm G}$-modules 
$S_{ij},\ ij \in {\rm V}$, given by $S_{ij} = k$ and $a1 = (-1)^i,\ b1 = (-1)^j$. 

Given $A$ and $(\alpha,\beta)$ as in the statement, we choose an isomorphism ${\rm G}\ \tilde{\to}\ \langle \alpha,\beta \rangle$ and compose it with the inclusion morphism $\langle \alpha,\beta \rangle 
\hookrightarrow {\rm Aut}(A)$, to obtain a group monomorphism $\rho: {\rm G} \hookrightarrow {\rm Aut}(A)$. This endows $A$ with the structure of a $k{\rm G}$-module, which decomposes into simple 
submodules
\[ A = U_1 \oplus U_2 \oplus U_3 \oplus U_4 = \bigoplus_{ij \in {\rm V}} B_{ij}, \]
where $B_{ij} = \bigoplus_{U_h \simeq S_{ij}}U_h$. For all $ij \in {\rm V}$,\vspace{0,1cm} the inclusion $B_{ij} \subset A_{ij}$ holds by definition of $B_{ij}$ and $A_{ij}$, whence $B_{ij} = A_{ij}$ follows
with Krull-Schmidt's Theorem. Thus, we proved the direct sum decomposition $A = \bigoplus_{ij \in {\rm V}} A_{ij}$. Straightforward arguments show that this is a 
{\rm V}-grading of $A$, such that $1 \in A_{00}$. It remains to show that $m_{ij} = \dim_k(A_{ij}) = 1$ for all $ij \in {\rm V}$.

Let $\chi_{ij}: {\rm G} \to k$ be the character of $S_{ij}$. Then the characters of the $k{\rm G}$-modules $A$ and $k{\rm G}$ are $\chi_A = \sum_{ij \in {\rm V}} m_{ij}\chi_{ij}$ and 
$\chi_{k{\rm G}} = \sum_{ij \in {\rm V}} \chi_{ij}$, respectively. Since the irreducible characters $\chi_{ij}$ are linearly independent over $k$ \cite[Theorem 30.12 (i)]{CR62}, it suffices to show that 
$\chi_A = \chi_{k{\rm G}}$.

So, let $g \in {\rm G} \setminus \{ e \}$ and set $E_g(\varepsilon) = E_{\rho(g)}(\varepsilon)$ for all $\varepsilon \in \{ \pm 1 \}$. Then $A = E_g(1) \oplus E_g(-1)$, because
$x = \frac{1}{2}(x + gx) + \frac{1}{2}(x - gx)$ holds for all $x \in A$. Now $g \not= e$ guarantees the existence 
of an element $u \in E_g(-1) \setminus \{0\}$. Since $A$ is a division algebra and $\rho(g) \in {\rm Aut}(A)$, the $k$-linear operator $L_u: A \to A,\ L_u(x) = ux$, 
is bijective and satisfies $L_u(E_g(1)) \subset E_g(-1)$ and $L_u(E_g(-1)) \subset E_g(1)$. It follows that $\dim_k(E_g(1)) = \dim_k(E_g(-1)) = 2$, and hence that $\chi_A(g) = 0 = \chi_{k{\rm G}}(g)$. 
Moreover, $\chi_A(e) = 4 = \chi_{k{\rm G}}(e)$. Altogether, we proved that $\chi_A = \chi_{k{\rm G}}$.
\end{proof}

\begin{pro} \label{trace bilinear form}
(i) For every division algebra $A \in \hat{\mathscr{C}}(k)$, the trace bilinear form $\tau_A: A \times A \to k,\ \tau_A(x,y) = {\rm tr}(L_{xy})$ is symmetric and non-degenerate. 
\\[1ex]
(ii) Every morphism $\varphi: A \to B$ in $\hat{\mathscr{C}}(k)$ is orthogonal with regard to $\tau_A$ and $\tau_B$, i.e. $\tau_A(x,y) = \tau_B(\varphi(x),\varphi(y))$ for all $x,y \in A$.
\end{pro}

\begin{proof}
{\it (i)} Every $A \in \hat{\mathscr{C}}(k)$ admits a {\rm V}-grading $A = \bigoplus_{ij \in {\rm V}} A_{ij}$\vspace{0,1cm} (Proposition \ref{V-grading}). Set $U = A_{01} \cup A_{10} \cup A_{11}$. 
Then ${\rm tr}(L_u) = 0$ for all $u \in U$. Choose a $k$-basis $(a_{ij})_{ij \in {\rm V}}$ in $A$ such that $a_{ij} \in A_{ij}$ for all $ij \in {\rm V}$. 
If $ij$ and $mn$ are distinct elements in ${\rm V}$, then $a_{ij}a_{mn} \in U$ and $a_{mn}a_{ij} \in U$ implies that $\tau_A(a_{ij},a_{mn}) = 0 = \tau_A(a_{mn},a_{ij})$, whence symmetry of $\tau_A$ follows. 

For every $x \in A \setminus \{0\}$ there exists an element $y \in A$ such that $xy = 1$. Then $\tau_A(x,y) = {\rm tr}(L_1) = 4 \not= 0$ shows that $\tau_A$ is non-degenerate.
\\[1ex]
{\it (ii)} Every morphism $\varphi: A \to B$ in $\hat{\mathscr{C}}(k)$ is a $k$-linear bijection. For all $z \in A$ and for every $k$-basis $\underline{a}$ in $A$, the matrix of $L_z$ in $\underline{a}$ 
equals the matrix of $L_{\varphi(z)}$ in $\varphi(\underline{a})$, whence ${\rm tr}(L_z) = {\rm tr}(L_{\varphi(z)})$ follows. Thus
\[ \tau_A(x,y) = {\rm tr}(L_{xy}) = {\rm tr}(L_{\varphi(xy)}) = {\rm tr}(L_{\varphi(x)\varphi(y)}) = \tau_B(\varphi(x),\varphi(y)) \]
holds for all $x,y \in A$.  
\end{proof}

\begin{lem} \label{orthogonality}
Let $A$ be a division algebra in $\hat{\mathscr{C}}(k)$, with Kleinian pair $(\alpha,\beta)$ and associated {\rm V}-grading $A = \bigoplus_{ij \in {\rm V}} A_{ij}$. Then,\vspace{0,1cm} for all 
$(ij,mn) \in {\rm V} \times {\rm V}$ and $(x_{ij},y_{mn}) \in A_{ij} \times A_{mn}$, the following holds true.
\\[1ex]
(i) $\tau_A(x_{ij},y_{mn}) = 0$\hspace{0,1cm} if\hspace{0,1cm} $ij \not= mn$. 
\\[1ex]
(ii) $\tau_A(x_{ij},y_{ij}) = 0$\hspace{0,1cm} if and only if\hspace{0,1cm} $x_{ij}y_{ij} = 0$.
\end{lem}

\begin{proof}
{\it (i)} The above proof of symmetry of $\tau_A$ contains even a proof of this statement. 
\\[1ex]
{\it (ii)} The ``if'' part holds trivially true. If $x_{ij}y_{ij} \not= 0$, then
$x_{ij}y_{ij} \in k1 \setminus \{0\}$ implies that $\tau_A(x_{ij},y_{ij}) \not= 0$.
\end{proof}

\begin{pro} \label{orthogonal supplement}
For every division algebra $A \in \hat{\mathscr{C}}(k)$ with Kleinian pair $(\alpha,\beta)$, the orthogonal supplement of $E_\beta(1)$ with regard to $\tau_A$ is 
\[ E_\beta(1)^\perp = E_\beta(-1). \]
\end{pro}

\begin{proof}
Let $A = \bigoplus_{ij \in {\rm V}} A_{ij}$ be the {\rm V}-grading associated with $(\alpha,\beta)$. Then $E_\beta(1) = A_{00} \oplus A_{10}$ and $E_\beta(-1) = A_{01} \oplus A_{11}$. The claimed identity 
is now a straightforward consequence of Lemma \ref{orthogonality}.
\end{proof}

\subsection{Reduction of objects in $\mathscr{C}(k)$}

Let $A$ be a division algebra in $\mathscr{C}(k)$. We say that a subfield $\ell \subset A$ is {\it Kleinian} if there exists a Kleinian pair $(\alpha,\beta)$ for $A$, such that 
$\ell = E_\beta(1) \subset N_r(A)$. Since $E_\beta(1) = A_{00} \oplus A_{10} = k1_A \oplus A_{10}$ (Proposition \ref{V-grading}), every Kleinian subfield of $\ell \subset A$ is a quadratic extension of $k$, such 
that $k \subset \ell \subset N_r(A)$.

\begin{theo} \label{reduction of objects}
For every division algebra $A \in \mathscr{C}(k)$, the following holds true.
\\[1ex]
(i) There exists a Kleinian subfield $\ell \subset A$. If $A$ is not associative, then $\ell = N_r(A)$ is the unique Kleinian subfield of $A$.
\\[1ex]
(ii) For every Kleinian subfield $\ell \subset A$, there exists an $\ell$-admissible triple $\underline{c} \in C_\ell$, such that $A(\ell,\underline{c})\ \tilde{\to}\ A$.
\end{theo}

\begin{proof}
{\it (i)} If $A$ is associative, choose any Kleinian pair $(\alpha,\beta)$ for $A$, and set $\ell = E_\beta(1)$. Then $\ell = A_{00} \oplus A_{10} = k1_A \oplus A_{10}$ (Proposition \ref{V-grading}) shows that 
$\ell$ is a 2-dimensional unital associative subalgebra of the division algebra $A$, contained in $A = N_r(A)$, and hence is a Kleinian subfield of $A$.

If $A$ is not associative, then we choose an injective group homomorphism $\rho: {\rm G} \hookrightarrow {\rm Aut}(A)$. Every automorphism $\varphi$ of $A$ induces an automorphism $\varphi_N$ of $N_r(A)$. 
This assignment is a group homomorphism 
\[ \nu: {\rm Aut}(A) \to {\rm Aut}(N_r(A)),\ \nu(\varphi) = \varphi_N. \] 
Due to Lemma \ref{nucleus} (iii) and Proposition \ref{decomposition}, $k \subset N_r(A)$ 
is a quadratic extension, and hence ${\rm Aut}(N_r(A)) = {\rm Gal}(N_r(A)/k)$ has order 2. The composed group homomorphism $\nu\rho: {\rm G} \to {\rm Aut}(N_r(A))$ is therefore not injective. 
Choose $b \in {\rm ker}(\nu\rho) \setminus \{e\}$ and $a \in {\rm G} \setminus \{e,b\}$. Then $(\alpha,\beta) = (\rho(a),\rho(b))$ is a Kleinian pair for $A$, and $\beta_N = \nu\rho(b) = \I_{N_r(A)}$
shows, together with\vspace{0,1cm} $\dim_k(E_\beta(1)) = 2 = \dim_k(N_r(A))$, that $E_\beta(1) = N_r(A)$. From this information we conclude with Proposition \ref{V-grading}, as in the associative case 
above, that $\ell = E_\beta(1)$ is a Kleinian subfield of $A$.  

If $A$ is not associative and $\ell \subset A$ is any Kleinian subfield, then the filtration $k \subset \ell \subset N_r(A)$, combined with $\dim_k(\ell) = 2 = \dim_k(N_r(A))$, yields $\ell = N_r(A)$. 
\\[1ex]
{\it (ii)} Given any Kleinian subfield $\ell \subset A$, there exists a Kleinian pair $(\alpha,\beta)$ for $A$, such that $\ell = E_\beta(1) \subset N_r(A)$.
Let $A = \bigoplus_{ij \in {\rm V}} A_{ij}$ be the {\rm V}-grading associated with $(\alpha,\beta)$. Then $\ell = A_{00} \oplus A_{10}$, $A_{00} = k1_A$ and $A_{10} = {\rm Im}(\ell)$. 
Choose $u \in A_{10} \setminus \{0\}$ and $v \in A_{01} \setminus \{0\}$. Then $\ell = k1_A \oplus ku,\ \overline{u} = -u$, and 
$A = \ell \oplus v\ell$.\vspace{0,1cm} 

The $k$-linear bijection $\varphi = L_v^{-1}R_v \in {\rm GL}(A)$ satisfies
\[ \varphi(1_A) = L_v^{-1}R_v(1_A) = L_v^{-1}L_v(1_A) = 1_A. \]
Since $uv \in A_{10}A_{01} = A_{11} = vA_{10}$, it also satisfies  
\[ \varphi(u) = L_v^{-1}R_v(u) = L_v^{-1}(uv) \in L_v^{-1}L_vA_{10} = A_{10}, \]
whence $\varphi(u) = \zeta_1 u$ for some $\zeta_1 \in k^\ast$. Thus, $\varphi \in {\rm GL}(A)$ induces a $k$-linear bijection $\varphi_\ell \in {\rm GL}(\ell)$. With $c_1 = \frac{1-\zeta_1}{2}$, 
the identity
\[ \varphi_\ell(x) = (1-c_1)x + c_1\overline{x} \]
holds for all $x \in \{1_A,u\}$, and hence for all $x \in \ell$. Accordingly,
\[ x(vz) = (xv)z = (L_vL_v^{-1}R_v(x))z = (v\varphi_\ell(x))z = v(\varphi_\ell(x)z) \]
yields
\begin{eqnarray}
x(vz) & = & v(((1-c_1)x + c_1\overline{x})z)
\end{eqnarray}
for all $x,z \in \ell$.

In the same vein, the $k$-linear bijection $\psi = R_vL_v \in {\rm GL}(A)$ satisfies $\psi(1_A) = (v1_A)v = vv = \zeta_21_A$ for some $\zeta_2 \in k^\ast$, and 
$\psi(u) = (vu)v = \zeta_3u$ for some $\zeta_3 \in k^\ast$, since $(vu)v \in (A_{01}A_{10})A_{01} = A_{11}A_{01} = A_{10} = ku$. Thus, $\psi \in {\rm GL}(A)$ induces a $k$-linear bijection 
$\psi_\ell \in {\rm GL}(\ell)$. With $c_2 = \frac{\zeta_2 + \zeta_3}{2}$ and $c_3 = \frac{\zeta_2 - \zeta_3}{2}$, 
the identity
\[ \psi_\ell(y) = c_2y + c_3\overline{y} \]
holds for all $y \in \{1_A,u\}$, and hence for all $y \in \ell$. Accordingly,
\begin{eqnarray} 
(vy)(vz) = ((vy)v)z = \psi_\ell(y)z = (c_2y + c_3\overline{y})z 
\end{eqnarray}
holds for all $y,z \in \ell$.

Summarizing, we have found a triple $\underline{c} = (c_1,c_2,c_3) \in k^3$, that satisfies (3) and (4) for all $x,y,z \in \ell$. We claim that the map 
\[ \theta: A(\ell,\underline{c})\ \tilde{\to}\ A,\ \theta \left( \begin{array}{c} x\\y \end{array} \right) = x+vy \] 
is an isomorphism of $k$-algebras. Indeed, $\theta$ is $k$-linear by definition, and bijective since $\ell \oplus v\ell = A$. Using (3) and (4), we see that
\begin{eqnarray*} 
\theta \left( \left( \begin{array}{c} x\\y \end{array} \right) \cdot \left( \begin{array}{c} w\\z \end{array} \right) \right) & = &  
\theta \left( \begin{array}{c} xw + (c_2y + c_3\overline{y})z \\ yw + ((1-c_1)x + c_1\overline{x})z \end{array} \right) 
\\[1ex]
 & = & xw + (c_2y + c_3\overline{y})z + v(yw + ((1-c_1)x + c_1\overline{x})z)
\\
 & = & xw + (vy)(vz) + (vy)w + x(vz)
\\
 & = & (x+vy)(w+vz) 
\\
 & = & \theta \left( \begin{array}{c} x\\y \end{array} \right) \theta \left( \begin{array}{c} w\\z \end{array} \right)
\end{eqnarray*}
holds for all $x,y,w,z \in \ell$. The algebra isomorphism $\theta: A(\ell,\underline{c})\ \tilde{\to}\ A$ is thus established. It follows that $A(\ell,\underline{c})$ is a division algebra, and hence that
$\underline{c}$ is $\ell$-admissible (Proposition \ref{constructed algebras} (iv)).
\end{proof}

\begin{cor} \label{Hurwitz division algebras}
For every field $k$ of characteristic not 2, the class $\mathscr{S}(k)$ of all 4-dimensional central skew fields over $k$ that admit a Kleinian pair coincides with the the class $\mathscr{HD}_4(k)$ of all 
4-dimensional Hurwitz division algebras over $k$.
\end{cor}

\begin{proof}
We already know that $\mathscr{S}(k) \supset \mathscr{HD}_4(k)$ (Proposition \ref{decomposition}). Conversely, for every $A \in \mathscr{S}(k)$ there exists a quadratic extension $k \subset \ell$ and an 
$\ell$-admissible triple $\underline{c} = (1,0,c_3) \in C_\ell^\mathscr{S}$, such that $A(\ell,\underline{c})\ \tilde{\to}\ A$\vspace{0,1cm} (\mbox{Theorem \ref{reduction of objects}} and 
Proposition \ref{types of admissible triples}). Identifying $(x,y) \in \ell^2$ with $\left( \begin{array}{c} x\\y \end{array}\right) \in \ell^{2 \times 1}$, the function $q_{\underline{c}}: \ell^2 \to \ell$ turns 
into a quadratic form
\[ q_{\underline{c}}: A(\ell,\underline{c}) \to k,\ q_{\underline{c}} \left( \begin{array}{c} x\\y \end{array}\right) = x\overline{x} - c_3y\overline{y}, \]
which is non-degenerate because it is anisotropic, and multiplicative because 
\[ q_{\underline{c}}(a \cdot b) = \det(L_{a \cdot b}) = \det(L_a L_b) = \det(L_a)\det(L_b) = q_{\underline{c}}(a)q_{\underline{c}}(b) \]
holds for all $a, b \in A(\ell,\underline{c})$. Accordingly, $A(\ell,\underline{c}) \in \mathscr{HD}_4(k)$. As $A\ \tilde{\to}\ A(\ell,\underline{c})$, it follows that $A \in \mathscr{HD}_4(k)$.
\end{proof}

\subsection{Reduction of morphisms in $\mathscr{C}(k)$}

Recall that all morphisms in $\mathscr{C}(k)$ are isomorphisms. We use the notation introduced in Subsection 3.3.

\begin{theo} \label{reduction of morphisms}
For every quadratic extension $k \subset \ell$ in characteristic \mbox{not 2} and for all $(\underline{c}, \underline{d}) \in C_\ell \times C_\ell$, the division
algebras $A(\ell,\underline{c})$ and $A(\ell,\underline{d})$ are isomorphic if and only if the subset $\ell^\ast(\underline{c},\underline{d}) \subset \ell^\ast$ is not empty.
\end{theo}

\begin{proof}
Let $k \subset \ell$ and $(\underline{c},\underline{d}) \in C_\ell \times C_\ell$ be given. If $\ell^\ast(\underline{c},\underline{d}) \not= \emptyset$, then we can choose 
$a \in \ell^\ast(\underline{c},\underline{d})$ and construct the isomorphism $\varphi_a: A(\ell,\underline{c}) \to A(\ell,\underline{d})$, defined in Lemma \ref{constructed morphisms} (ii).

Conversely, suppose that an isomorphism \mbox{$\mu: A(\ell,\underline{c}) \to A(\ell,\underline{d})$ exists.} Towards a proof of non-emptiness of $\ell^\ast(\underline{c},\underline{d})$ we proceed in 
four steps, wherein we denote the common unity $1_{A(\ell,\underline{c})} = 1_{A(\ell,\underline{d})} = \left( \begin{array}{c} 1\\0 \end{array} \right)$ by $1_A$, and set 
$j = \left( \begin{array}{c} 0\\1 \end{array} \right)$.
\\[1ex]
{\it Step 1. There exists an isomorphism $\mu_1: A(\ell,\underline{c}) \to A(\ell,\underline{d})$, such that $\mu_1(1_A\ell) = 1_A\ell$.} 
\\[1ex]
{\it Proof of Step 1.} The existence of $\mu$ implies with Proposition \ref{decomposition}, that
\[ (\underline{c},\underline{d}) \in \left( C_\ell^{\mathscr{N}} \times C_\ell^{\mathscr{N}} \right) \sqcup \left( C_\ell^{\mathscr{S}} \times C_\ell^{\mathscr{S}} \right) \sqcup 
\left( C_\ell^{\mathscr{K}} \times C_\ell^{\mathscr{K}} \right). \]
Suppose $(\underline{c},\underline{d}) \in C_\ell^{\mathscr{N}} \times C_\ell^{\mathscr{N}}$. The isomorphism $\mu$ induces an isomorphism of right nuclei, and 
$N_r(A(\ell,\underline{c})) = 1_A\ell = N_r(A(\ell,\underline{d}))$ by Proposition \ref{constructed algebras} (ii) and Proposition \ref{decomposition}, so $\mu(1_A\ell) = 1_A\ell$ follows.
\\[1ex]
Suppose $(\underline{c},\underline{d}) \in C_\ell^{\mathscr{S}} \times C_\ell^{\mathscr{S}}$. The isomorphism $\mu$ induces an isomorphism $\mu_\iota: 1_A\ell \to \mu(1_A\ell)$ of subfields of $A(\ell,\underline{d})$.
Since $A(\ell,\underline{d})$ is a central skew field over $k$, \cite[Theorem 7.21]{Re75} asserts that $\mu_\iota$ extends to an automorphism $\vartheta$ of $A(\ell,\underline{d})$. Now
$\mu_1 = \vartheta^{-1}\mu$ will do.
\\[1ex]
Suppose $(\underline{c},\underline{d}) \in C_\ell^{\mathscr{K}} \times C_\ell^{\mathscr{K}}$. Choose $u \in {\rm Im}(1_A \ell) \setminus \{0\}$. Then $u^2 = 1_A \lambda$ for some $\lambda \in k \setminus k_{sq}$. Now
$(\mu(u))^2 = \mu(u^2) = \mu(1_A \lambda) = 1_A \lambda = u^2$ implies, since $A(\ell,\underline{d})$ is a field, that $\mu(u) = \pm u$. Because $1_A \ell = 1_A k \oplus uk$, it follows that 
$\mu(1_A \ell) = 1_A \ell$.   
\\[1ex]
{\it Step 2. There exists an isomorphism $\mu_2: A(\ell,\underline{c}) \to A(\ell,\underline{d})$ that fixes $1_A\ell$ elementwise.} 
\\[1ex]
{\it Proof of Step 2.} Let $\mu_1: A(\ell,\underline{c}) \to A(\ell,\underline{d})$ be an isomorphism as in Step 1. Then $\mu_1$ induces a Galois automorphism $\mu_\ell \in {\rm Gal}(\ell/k)$, 
and ${\rm Gal}(\ell/k) = \langle \sigma \rangle$ has order 2. If $\mu_\ell = \I_\ell$, then $\mu_2 = \mu_1$ will do. If $\mu_\ell = \sigma$, then $\mu_2 = \mu_1\psi_1$  will do, where
$\psi_1 \in {\rm Aut}(A(\ell,\underline{c}))$ is the automorphism defined in Lemma \ref{constructed morphisms} (ii).
\\[1ex]
{\it Step 3. If $\mu_2: A(\ell,\underline{c}) \to A(\ell,\underline{d})$ is an isomorphism that fixes $1_A\ell$ elementwise, then $\mu_2(j) = ja$ for some $a \in \ell^\ast$.}
\\[1ex]
{\it Proof of Step 3.} Let $\mu_2: A(\ell,\underline{c}) \to A(\ell,\underline{d})$ be an isomorphism that fixes $1_A\ell$ elementwise, and let $(\beta,\alpha)$ be the Kleinian pair for 
$A(\ell,\underline{c})$ and $A(\ell,\underline{d})$, defined in Proposition \ref{constructed algebras} (iii). Then, by Proposition \ref{orthogonal supplement}, the orthogonal supplement of $1_A\ell$ 
with regard to the trace bilinear form is 
\[ 1_A\ell^\perp = E_\alpha(1)^\perp = E_\alpha(-1) = j\ell \]
in both $A(\ell,\underline{c})$ and $A(\ell,\underline{d})$. We conclude with Proposition \ref{trace bilinear form} (ii) that 
\[ \tau_{A(\ell,\underline{d})}(1_Al,\mu_2(j)) = \tau_{A(\ell,\underline{d})}(\mu_2(1_Al),\mu_2(j)) = \tau_{A(\ell,\underline{c})}(1_Al,j) = 0 \]
holds for all $l \in \ell$. So $\mu_2(j) \in 1_A\ell^\perp = j\ell$, i.e.~$\mu_2(j) = ja$ for some $a \in \ell^\ast$.
\\[1ex]
{\it Step 4. If $\mu_2: A(\ell,\underline{c}) \to A(\ell,\underline{d})$ and $a \in \ell^\ast$ are as in Step 3, then $a \in \ell^\ast(\underline{c},\underline{d})$.} 
\\[1ex]
{\it Proof of Step 4.} As $\mu_2$ fixes $1_A\ell$ elementwise, identity (2) in the proof of Proposition \ref{constructed algebras} shows that $\mu_2$ is right $\ell$-linear. Thus, the system of equations
\[ \left\{ \begin{array}{ccc} 
\mu_2(1_Ax \cdot j) & = & 1_Ax \cdot ja \\
\mu_2(jy \cdot j) & = & \mu_2(jy) \cdot ja,
\end{array} \right. \]
valid for all $(x,y) \in \ell^2$, gives rise to the system of equations
\[ \left\{ \begin{array}{ccc} 
(1-c_1)x + c_1 \overline{x} & = & (1-d_1)x + d_1 \overline{x} 
\\[1ex]
c_2y + c_3 \overline{y} & = & a^2d_2y + a\overline{a}d_3 \overline{y},
\end{array} \right. \]
which, evaluated in $x = i \in {\rm Im}(\ell) \setminus \{0\}$ and $y \in \{ 1,i \}$, yields $c_1 = d_1$ and $(c_2, c_3) = (a^2d_2, a\overline{a}d_3)$, i.e.~$a \in \ell^\ast (\underline{c}, \underline{d})$. 
\end{proof}

\subsection{$\mathscr{L}$-coverings of $\mathscr{C}(k)$ and $\mathscr{N}(k), \mathscr{S}(k), \mathscr{K}(k)$}

For every division algebra $A \in \mathscr{C}(k)$, we denote by ${\cal N}(A)$ the set of all Kleinian subfields of $A$. Then ${\cal N}(A) \subset \mathscr{Q}(k)$, where $\mathscr{Q}(k)$ is the groupoid 
formed by all quadratic extensions of $k$ (Subsection 2.1). Let $\mathscr{L} \subset \mathscr{Q}(k)$ be a transversal for the set $\mathscr{Q}(k)/\hspace{-0,1cm}\simeq$ of all isoclasses in $\mathscr{Q}(k)$. 
For each $\ell \in \mathscr{L}$ we define the full subgroupoid $\mathscr{C}_\ell(k) \subset \mathscr{C}(k)$ by its object class
\[ \mathscr{C}_\ell(k) = \{ A \in \mathscr{C}(k)\ |\ \ell\ \tilde{\to}\ n\ \mbox{for\ some}\ n \in {\cal N}(A) \}, \]
and likewise we define the full subgroupoids $\mathscr{N}_\ell(k), \mathscr{S}_\ell(k)$ and $\mathscr{K}_\ell(k)$ respectively of $\mathscr{N}(k), \mathscr{S}(k)$ and $\mathscr{K}(k)$ by their object classes
\[ \mathscr{N}_\ell(k) = \mathscr{N}(k) \cap \mathscr{C}_\ell(k),\ \mathscr{S}_\ell(k) = \mathscr{S}(k) \cap \mathscr{C}_\ell(k)\ \mbox{and}\ \mathscr{K}_\ell(k) = \mathscr{K}(k) \cap \mathscr{C}_\ell(k). \]

\begin{pro} \label{L-coverings}
The groupoid $\mathscr{C}(k)$ and its blocks $\mathscr{N}(k), \mathscr{S}(k), \mathscr{K}(k)$ admit coverings of the form
\begin{eqnarray*} \mathscr{C}(k) & = & \bigcup_{\ell \in \mathscr{L}} \mathscr{C}_\ell(k) = \bigcup_{\ell \in \mathscr{L}} \left( \mathscr{N}_\ell(k) \amalg \mathscr{S}_\ell(k) \amalg \mathscr{K}_\ell(k) \right),\\ 
\mathscr{N}(k) & = & \coprod_{\ell \in \mathscr{L}} \mathscr{N}_\ell(k),\\
\mathscr{S}(k) & = & \bigcup_{\ell \in \mathscr{L}} \mathscr{S}_\ell(k),\\
\mathscr{K}(k) & = & \bigcup_{\ell \in \mathscr{L}} \mathscr{K}_\ell(k). 
\end{eqnarray*}
\end{pro}

\begin{proof}
Every $A \in \mathscr{C}(k)$ has a Kleinian subfield $n \subset A$ (Theorem \ref{reduction of objects} (i)), and $n\ \tilde{\to}\ \ell$ for some $\ell \in \mathscr{L}$, so $A \in  \mathscr{C}_\ell(k)$.
This proves the covering statement for $\mathscr{C}(k)$. Taking $A$ in $\mathscr{N}(k), \mathscr{S}(k)$ and $\mathscr{K}(k)$ respectively, we obtain the covering statements for 
$\mathscr{N}(k), \mathscr{S}(k)$ and $\mathscr{K}(k)$. For all $\ell \in \mathscr{L}$, the decomposition $\mathscr{C}_\ell(k) = \mathscr{N}_\ell(k) \amalg \mathscr{S}_\ell(k) \amalg \mathscr{K}_\ell(k)$ holds
by Proposition \ref{decomposition}.

Let $\ell, m \in \mathscr{L}$. If $A \in \mathscr{N}_\ell(k)$ and $B \in \mathscr{N}_m(k)$ are isomorphic objects, then Theorem \ref{reduction of objects} (i), second statement, yields
$\ell\ \tilde{\to}\ N_r(A)\ \tilde{\to}\ N_r(B)\ \tilde{\to}\ m$. This implies $\ell = m$, whence the decomposition statement for $\mathscr{N}(k)$ follows. 
\end{proof}

\noindent
The question of how to find a transversal $\mathscr{L}$ for $\mathscr{Q}(k)/\hspace{-0,1cm}\simeq$ arises naturally in the context of Proposition \ref{L-coverings}. It is answered on the level of 
the punctured quotient group $\left( k^\ast/k_{sq}^\ast \right)^\circ$, introduced in Subsection 2.1, by the following proposition, whose proof makes an exercise in basic algebra.

\begin{pro} \label{quadratic extensions}
For every field $k$ of characteristic not 2, the following statements hold true.
\\[1ex]
(i) If $a \in k^\ast \setminus k_{sq}^\ast$, then $k (\sqrt{a}~) \in \mathscr{Q}(k)$.
\\[1ex]
(ii) If $\ell \in \mathscr{Q}(k)$, then $\ell\ = k(\sqrt{a}~)$, for some $a \in k^\ast \setminus k_{sq}^\ast$.
\\[1ex]
(iii) For all $a, b \in k^\ast \setminus k_{sq}^\ast$, the coset identity $ak_{sq}^\ast = bk_{sq}^\ast$ holds if and 
only if $k(\sqrt{a}~)\ \tilde{\to}\ k(\sqrt{b}~)$.
\\[1ex]
(iv) A subset $L \subset k^\ast \setminus k_{sq}^\ast$ is a transversal for $\left( k^\ast/k_{sq}^\ast \right)^\circ$\vspace{0,1cm} if and only if the subset 
$\mathscr{L} = \left\{ \left. k(\sqrt{a}~) \right| a \in L \right\} \subset \mathscr{Q}(k)$ is a transversal for $\mathscr{Q}(k)/\hspace{-0,1cm}\simeq$. 
\\[1ex]
(v) The order of the quotient group $k^\ast/k_{sq}^\ast$ and the cardinality\vspace{0,1cm} of the set of isoclasses
$\mathscr{Q}(k)/\hspace{-0,1cm}\simeq$ satisfy the identity $\left| k^\ast/k_{sq}^\ast \right| - 1 = \left| \mathscr{Q}(k)/\hspace{-0,1cm}\simeq \right|$. 
\end{pro}

\section{Classifications}

We retain our assumption, that $k$ is a field of characteristic not 2. The present section is devoted to the problem of classifying the groupoid $\mathscr{C}(k)$. This problem is reduced in the first 
instance, by Proposition \ref{L-coverings}, to the problem of classifying $\mathscr{C}_\ell(k)$ for all $\ell$ in a transversal $\mathscr{L}$ for $\mathscr{Q}(k)/\hspace{-0,1cm}\simeq$, and 
such a transversal can be found in accordance with Proposition \ref{quadratic extensions} (iv). For each $\ell \in \mathscr{L}$, the problem of classifying $\mathscr{C}_\ell(k)$ is in turn reduced, by 
Theorems \ref{reduction of objects} and \ref{reduction of morphisms}, to the problem of classifying all $\ell$-admissible triples up to $\ell^\ast$-equivalence. The latter problem, when restricted to 
$C_\ell^{\mathscr{S}} \sqcup C_\ell^{\mathscr{K}}$, can be solved on the level of two punctured quotient groups of $k^\ast$.

These reductions, which we here only outlined, are made precise in Subsection 5.1 below, and explicit classifications will be derived from them in Subsections 5.3-5.4.

\subsection{Reduction of the classification problem of $\mathscr{C}(k)$}

Let $k \subset \ell$ be a quadratic extension in characteristic not 2. On the set $C_\ell$ of all $\ell$-admissible triples we introduce a binary relation $\sim$, defined by
$\underline{c} \sim \underline{d}$ if and only if 
\mbox{$\ell^\ast(\underline{c},\underline{d}) \not= \emptyset$.} It is in fact an equivalence relation, called {\it $\ell^\ast$-equivalence} on $C_\ell$. Its equivalence classes refine the 
partition $C_\ell = C_\ell^{\mathscr{N}} \sqcup C_\ell^{\mathscr{S}} \sqcup C_\ell^{\mathscr{K}}$, introduced in Subsection 3.2. The set of all $\ell^\ast$-equivalence classes of 
$C_\ell, C_\ell^{\mathscr{N}}, C_\ell^{\mathscr{S}}$ and $C_\ell^{\mathscr{K}}$ respectively will be denoted by $C_\ell/\ell^\ast, C_\ell^{\mathscr{N}}/\ell^\ast, C_\ell^{\mathscr{S}}/\ell^\ast$ and $C_\ell^{\mathscr{K}}/\ell^\ast$.

We proceed with the following improvement of Corollary \ref{constructed division algebras}.

\begin{cor} \label{constructed l-division algebras}
If $\underline{c} \in C_\ell$, then $A(\ell,\underline{c}) \in \mathscr{C}_\ell(k)$.
\end{cor}

\begin{proof}
If $\underline{c} \in C_\ell$, then $A(\ell,\underline{c}) \in \mathscr{C}(k)$, by Corollary \ref{constructed division algebras}. Proposition \ref{constructed algebras} (ii)-(iii) shows that 
the subfield $1_A\ell \subset A(\ell,\underline{c})$ is Kleinian and is isomorphic to $\ell$. So 
$A(\ell,\underline{c}) \in \mathscr{C}_\ell(k)$.
\end{proof}

\noindent
It is thus justified to view the map $\mathscr{F}_\ell: C_\ell \to \mathscr{C}$, introduced subsequent to Corollary \ref{constructed division algebras}, as a map 
\[ \mathscr{F}_\ell: C_\ell \to \mathscr{C}_\ell,\ \mathscr{F}_\ell(\underline{c}) = A(\ell,\underline{c}), \] 
where $\mathscr{C}_\ell = \mathscr{C}_\ell(k)$. It induces the maps\vspace{0,1cm}
$\mathscr{F}_\ell: C_\ell^{\mathscr{N}} \to \mathscr{N}_\ell, \mathscr{F}_\ell: C_\ell^{\mathscr{S}} \to \mathscr{S}_\ell$ and $\mathscr{F}_\ell: C_\ell^{\mathscr{K}} \to \mathscr{K}_\ell$, 
where $\mathscr{N}_\ell = \mathscr{N}_\ell(k), \mathscr{S}_\ell = \mathscr{S}_\ell(k)$ and $\mathscr{K}_\ell = \mathscr{K}_\ell(k)$. 

\begin{theo} \label{reduction of classifications}
For every quadratic extension $k \subset \ell$ in characteristic not 2, and for all subsets $T_\ell, T_\ell^{\mathscr{N}}, T_\ell^{\mathscr{S}}$ and $T_\ell^{\mathscr{K}}$ of $C_\ell, C_\ell^{\mathscr{N}}, C_\ell^{\mathscr{S}}$
and $C_\ell^{\mathscr{K}}$ respectively, the following statements hold true.
\\[1ex]
(i) $T_\ell$ is a transversal for $C_\ell/\ell^\ast$ if and only if $\mathscr{F}_\ell(T_\ell)$ classifies $\mathscr{C}_\ell$.
\\[1ex]
(ii) $T_\ell^{\mathscr{N}}$ is a transversal for $C_\ell^{\mathscr{N}}/\ell^\ast$ if and only if $\mathscr{F}_\ell(T_\ell^{\mathscr{N}})$ classifies $\mathscr{N}_\ell$.
\\[1ex]
(iii) $T_\ell^{\mathscr{S}}$ is a transversal for $C_\ell^{\mathscr{S}}/\ell^\ast$ if and only if $\mathscr{F}_\ell(T_\ell^{\mathscr{S}})$ classifies $\mathscr{S}_\ell$.
\\[1ex]
(iv) $T_\ell^{\mathscr{K}}$ is a transversal for $C_\ell^{\mathscr{K}}/\ell^\ast$ if and only if $\mathscr{F}_\ell(T_\ell^{\mathscr{K}})$ classifies $\mathscr{K}_\ell$.
\end{theo}

\begin{proof}
{\it (i)} Let $k \subset \ell$ and $T_\ell \subset C_\ell$ be given. Because the equivalences 
\[ \underline{c} \sim \underline{d}\ \Leftrightarrow\ \ell^\ast(\underline{c},\underline{d}) \not= \emptyset\ \Leftrightarrow\  A(\ell,\underline{c})\ \tilde{\to}\ A(\ell,\underline{d}) \]
hold for all $\underline{c},\underline{d} \in T_\ell$ (Theorem \ref{reduction of morphisms}), the subset $T_\ell \subset C_\ell$ is irredundant regarding $\ell^\ast$-equivalence if and only if the 
subset $\mathscr{F}_\ell(T_\ell) \subset \mathscr{C}_\ell$ is irredundant regarding isomorphism. 

Suppose that $T_\ell$ exhausts $C_\ell/\ell^\ast$, and let $A \in \mathscr{C}_\ell$. Then $A$ has a Kleinian subfield $n$, such that $\ell\ \tilde{\to}\ n$. Theorem \ref{reduction of objects} (ii) guarantees the 
existence of an $n$-admissible triple $\underline{c} \in C_n$, such that $A(n,\underline{c})\ \tilde{\to}\ A$. Moreover, $\ell\ \tilde{\to}\ n$ implies that $\underline{c} \in C_\ell$ and 
$A(\ell,\underline{c})\ \tilde{\to}\ A(n,\underline{c})$. By hypothesis on $T_\ell$, the triple $\underline{c} \in C_\ell$ is $\ell^\ast$-equivalent to some $\underline{d} \in T_\ell$.
Applying Theorem \ref{reduction of morphisms}, we arrive at the chain of isomorphisms $\mathscr{F}_\ell(\underline{d}) = A(\ell,\underline{d})\ \tilde{\to}\ A(\ell,\underline{c})\ \tilde{\to}\ 
A(n,\underline{c})\ \tilde{\to}\ A$, which proves that $\mathscr{F}_\ell(T_\ell)$ exhausts $\mathscr{C}_\ell/\hspace{-0,1cm}\simeq$.

Conversely, suppose that $\mathscr{F}_\ell(T_\ell)$ exhausts $\mathscr{C}_\ell/\hspace{-0,1cm}\simeq$, and let $\underline{c} \in C_\ell$. Then $\mathscr{F}_\ell(\underline{c}) \in \mathscr{C}_\ell$ (Corollary 
\ref{constructed l-division algebras}). By hypothesis on $T_\ell$, the division algebra $\mathscr{F}_\ell(\underline{c})$ is isomorphic to $\mathscr{F}_\ell(\underline{d})$, for some $\underline{d} \in T_\ell$.
With Theorem \ref{reduction of morphisms} we deduce the $\ell^\ast$-equivalence $\underline{c} \sim \underline{d}$, which proves that $T_\ell$ exhausts $C_\ell/\ell^\ast$.
\\[1ex]
{\it (ii)-(iv)} These statements are proved analogously.
\end{proof}

\noindent
Transversals for $C_\ell^{\mathscr{S}}/\ell^\ast$ and $C_\ell^{\mathscr{K}}/\ell^\ast$ can be found as follows.

\begin{pro} \label{transversals of punctured quotient groups}
For every quadratic extension $k \subset \ell$ in characteristic \mbox{not 2,} the following holds true.
\\[1ex]
(i) $C_\ell^\mathscr{S} = \{ (1,0,c_3)\ |\ c_3 \in k^\ast \setminus {\rm im}(n_{\ell/k}^\ast) \}$. 
\\[1ex]
(ii) A subset $S_\ell \subset k^\ast \setminus {\rm im}(n_{\ell/k}^\ast)$ is a transversal for $(k^\ast/{\rm im}(n_{\ell/k}^\ast))^\circ$ if and only if $\hat{S}_\ell =  \{ (1,0,c_3)\ |\ c_3 \in S_\ell \}$ 
is a transversal for $C_\ell^\mathscr{S}/\ell^\ast$.
\\[1ex]
(iii) $C_\ell^\mathscr{K} = \{ (0,c_2,0)\ |\ c_2 \in k^\ast \setminus \ell_{sq}^\ast \}$.
\\[1ex]
(iv) A subset $K_\ell \subset k^\ast \setminus \ell_{sq}^\ast$ is a transversal for $(k^\ast/(\ell_{sq}^\ast \cap k^\ast))^\circ$ if and only if $\hat{K}_\ell =  \{ (0,c_2,0)\ |\ c_2 \in K_\ell \}$ 
is a transversal for $C_\ell^\mathscr{K}/\ell^\ast$.
\end{pro}

\begin{proof}
{\it (i)} By Proposition \ref{types of admissible triples} (ii), $C_\ell^\mathscr{S}$ consists of all triples $\underline{c} = (1,0,c_3) \in k^3$ such that the associated function 
$q_{\underline{c}}: \ell^2 \to \ell,\ q_{\underline{c}}(x,y) = x\overline{x} - c_3 y\overline{y}$ is anisotropic. This occurs if and only if $c_3 \in k^\ast \setminus {\rm im}(n_{\ell/k}^\ast)$. 
\\[1ex]
{\it (ii)} Let $S_\ell \subset k^\ast \setminus {\rm im}(n_{\ell/k}^\ast)$ be given. Then, for all $c, d \in S_\ell$, the\vspace{0,1cm} \mbox{coset identity} $c({\rm im}(n_{\ell/k}^\ast)) = d({\rm im}(n_{\ell/k}^\ast))$ holds 
if and only if $(1,0,c)$ and $(1,0,d)$ are \mbox{$\ell^\ast$-equivalent.} Together with (i), this establishes the claimed equivalence.
\\[1ex]
{\it (iii)} By Proposition \ref{types of admissible triples} (iii), $C_\ell^\mathscr{K}$ consists of all triples $\underline{c} = (0,c_2,0) \in k^3$ such that the associated function 
$q_{\underline{c}}: \ell^2 \to \ell,\ q_{\underline{c}}(x,y) = x^2 - c_2 y^2$ is anisotropic. This occurs if and only if $c_2 \in k^\ast \setminus \ell_{sq}^\ast$.
\\[1ex]
{\it (iv)} Let $K_\ell \subset k^\ast \setminus \ell_{sq}^\ast$ be given. Then, for all $c, d \in K_\ell$, the\vspace{0,1cm} coset identity\linebreak $c(\ell_{sq}^\ast \cap k^\ast) = d(\ell_{sq}^\ast \cap k^\ast)$ holds if and only if 
$(0,c,0)$ and $(0,d,0)$ are $\ell^\ast$-equi\-valent. Together with (iii), this establishes the claimed equivalence.
\end{proof}

\noindent
Theorem \ref{reduction of classifications} (iii)-(iv) and Proposition \ref{transversals of punctured quotient groups} have the following immediate consequence. 

\begin{cor} \label{one-parameter classifications}
For every quadratic extension $k \subset \ell$ in characteristic \mbox{not 2,} the following holds true.
\\[1ex]
(i) A subset $S_\ell \subset k^\ast \setminus {\rm im}(n_{\ell/k}^\ast)$ is a transversal for $(k^\ast/{\rm im}(n_{\ell/k}^\ast))^\circ$ if and only if the associated set of central skew fields
$\{ A(\ell,(1,0,c_3))\ |\ c_3 \in S_\ell \}$ \mbox{classifies $\mathscr{S}_\ell$.}
\\[1ex]
(ii) A subset $K_\ell \subset k^\ast \setminus \ell_{sq}^\ast$ is a transversal for $(k^\ast/(\ell_{sq}^\ast \cap k^\ast))^\circ$ if and only if the associated set of fields
$\{ A(\ell,(0,c_2,0))\ |\ c_2 \in K_\ell \}$ classifies $\mathscr{K}_\ell$.
\end{cor}

\noindent
Let us summarize our present insight into the classification problem of $\mathscr{C}(k)$, where $k$ is any field of characteristic not 2. With support of Propositions \ref{quadratic extensions} and
\ref{transversals of punctured quotient groups} we may assume that a transversal $\mathscr{L}$ for $\mathscr{Q}(k)/\hspace{-0,1cm}\simeq$ is known and that, for each $\ell \in \mathscr{L}$, 
transversals $\hat{S}_\ell$ for $C_\ell^\mathscr{S}/\ell^\ast$ and $\hat{K}_\ell$ for $C_\ell^\mathscr{K}/\ell^\ast$ are known. Then $\bigcup_{\ell \in \mathscr{L}} \mathscr{F}_\ell(\hat{S}_\ell)$ exhausts 
$\mathscr{S}(k)/\hspace{-0,1cm}\simeq$ and $\bigcup_{\ell \in \mathscr{L}} \mathscr{F}_\ell(\hat{K}_\ell)$ exhausts $\mathscr{K}(k)/\hspace{-0,1cm}\simeq$, by Proposition \ref{L-coverings} and
Theorem \ref{reduction of classifications} (iii)-(iv). In order to accomplish a classification of $\mathscr{C}(k)$ on this basis, it suffices in view of Propositions \ref{decomposition} and \ref{L-coverings}
and Theorem \ref{reduction of classifications} (ii), to solve the following four remaining  problems. 
\\[1ex]
{\bf Problem 1.} {\it For each $\ell \in \mathscr{L}$, display the subset $C_\ell^\mathscr{N} \subset k^3$ explicitly.}
\\[1ex]
{\bf Problem 2.} {\it For each $\ell \in \mathscr{L}$, find a transversal $T_\ell^\mathscr{N}$ for $C_\ell^\mathscr{N}/\ell^\ast$.}
\\[1ex]
{\bf Problem 3.} {\it Find a subset $T^\mathscr{S} \subset \bigsqcup_{\ell \in \mathscr{L}} \left( \{\ell\} \times \hat{S}_\ell \right)$ such that the associated subset  
$\left\{ A(\ell,\underline{c})\ |\ (\ell,\underline{c}) \in T^\mathscr{S} \right\} \subset \bigcup_{\ell \in \mathscr{L}} \mathscr{F}_\ell(\hat{S}_\ell)$ classifies $\mathscr{S}(k)$.}
\\[1ex]
{\bf Problem 4.} {\it Find a subset $T^\mathscr{K} \subset \bigsqcup_{\ell \in \mathscr{L}} \left( \{\ell\} \times \hat{K}_\ell \right)$ such that the associated subset  
$\left\{ A(\ell,\underline{c})\ |\ (\ell,\underline{c}) \in T^\mathscr{K} \right\} \subset \bigcup_{\ell \in \mathscr{L}} \mathscr{F}_\ell(\hat{K}_\ell)$ classifies $\mathscr{K}(k)$.}
\\[1ex]
The complexity of Problems 1-4 depends heavily on the nature of the ground field $k$. It depends, in particular, on the cardinality $|\mathscr{Q}(k)/\hspace{-0,1cm}\simeq| = |k^\ast/k_{sq}^\ast| - 1$ 
(Proposition \ref{quadratic extensions} (v)).

The trivial case occurs when $\mathscr{Q}(k) = \emptyset$, or equivalently, when $k^\ast = k_{sq}^\ast$. Then Problems 1-4 are empty, and so is $\mathscr{C}(k)$.

The simplest non-trivial case occurs when $|\mathscr{Q}(k)/\hspace{-0,1cm}\simeq| = 1$, or equivalently, when $|k^\ast/k_{sq}^\ast| = 2$ . We express this situation by saying that $k$ has 
{\it unique quadratic extension class}. For such fields $k$, Problems 1-4 admit a further reduction, expounded in Subsection 5.2 below. On that basis, we achieve in 
Subsections 5.3 and 5.4 explicit classifications of $\mathscr{C}(k)$ for two types of ground fields $k$ having unique quadratic extension class, namely for all ordered fields $k$ in which every 
positive element is a square, and for all finite fields $k$ of odd order, respectively. 

In the arithmetic case $k = \Q$ we have that $|\mathscr{Q}(\Q)/\hspace{-0,1cm}\simeq| = \aleph_0$, and complete solutions to Problems 1-4 seem not to be known at present. Interesting partial solutions, 
with a distinct flavour of classic number theory, have recently been achieved by G.~Hammarhjelm \cite{Ha17}.

\subsection{Fields with unique quadratic extension class}

In this subsection, we study the decomposition $ \mathscr{C}(k) = \mathscr{N}(k) \amalg \mathscr{S}(k) \amalg \mathscr{K}(k)$ (Proposition \ref{decomposition}) in case $k$ has characteristic not 2 and 
unique quadratic extension class. It turns out that, for all such fields $k$, the block $\mathscr{K}(k)$ is empty, and the block $\mathscr{S}(k)$ consists of at most one isoclass
(Proposition \ref{dichotomy}). As a consequence, the problem of classifying $\mathscr{C}(k)$ is reduced to a simplified version of the above Problems 1-4, formulated below as Problems $1'$-$3'$. We begin
with two preparatory lemmas.

\begin{lem} \label{three subgroups}
Let $k \subset \ell$ be a quadratic extension in characteristic not 2. Choose $i \in {\rm Im}(\ell) \setminus \{0\}$, 
and set $i^2 = -t$. Then the following holds true.
\\[1ex]
(i) ${\rm im}(n_{\ell/k}^\ast) = \{ x^2 + ty^2\ |\ (x,y) \in k^2 \setminus \{(0,0)\} \}$.
\\[1ex]
(ii) $\ell_{sq}^\ast \cap k^\ast = k_{sq}^\ast \sqcup -tk_{sq}^\ast$.
\\[1ex]
(iii) $k_{sq}^\ast < {\rm im}(n_{\ell/k}^\ast) < k^\ast$.
\\[1ex]
(iv) $k_{sq}^\ast < \ell_{sq}^\ast \cap k^\ast < k^\ast$ and $\left| (\ell_{sq}^\ast \cap k^\ast)/k_{sq}^\ast \right| = 2$.
\end{lem}

\begin{proof}
For every $a \in \ell^\ast$ there is a unique pair $(x,y) \in k^2 \setminus \{(0,0)\}$, such that $a = x + iy$. Now 
$\overline{a} = x - iy$ implies $a\overline{a} = x^2 + ty^2$, which proves (i). Also, $a^2 \in k^\ast$ if and only if 
$xy = 0$, and in that case $a^2 = x^2$ or $a^2 = -ty^2$. Since $-t \not\in k_{sq}^\ast$, this proves (ii). Setting 
$y = 0$ in (i) one obtains (iii),\vspace{0,1cm} while (ii) implies (iv). 
\end{proof}

\begin{lem} \label{norm and squares}
Let $k$ be a field of characteristic not 2 that has unique quadratic extension class, represented by $\ell$. Then the 
following holds true.
\\[1ex]
(i) If $n_{\ell/k}^\ast: \ell^\ast \to k^\ast$ is not surjective, then ${\rm im}(n_{\ell/k}^\ast) = 
k_{sq}^\ast$.
\\[1ex]
(ii) $\ell_{sq}^\ast \cap k^\ast = k^\ast$.
\end{lem}

\begin{proof}
As $|k^\ast/k_{sq}^\ast| = 2$ holds by hypothesis, (i) and (ii) follow directly from Lemma \ref{three subgroups} (iii) and (iv).
\end{proof}

\begin{pro} \label{dichotomy}
Let $k$ be a field of characteristic not 2 that has unique quadratic extension class, represented by $\ell$. Then the 
following holds true.
\\[1ex]
(i)\hspace{0,1cm} $\mathscr{K}(k) = \emptyset$.
\\[1ex]
(ii) If $n_{\ell/k}^\ast$ is not surjective, then $\mathscr{S}(k)$ consists of precisely one isoclass, represented by $A(\ell, (1,0,c_3))$ for any $c_3 \in k^\ast \setminus {\rm im}(n_{\ell/k}^\ast)$.
\\[1ex]
(iii) If $n_{\ell/k}^\ast$ is surjective, then $\mathscr{S}(k) = \emptyset$. Moreover, 
\[ C_\ell = C_\ell^\mathscr{N} \subset \{ \underline{c} \in k^3\ |\ c_2 \not= 0 \}. \]
\end{pro}

\begin{proof}
{\it (i)} Lemma \ref{norm and squares} (ii) asserts that $(k^\ast/(\ell_{sq}^\ast \cap k^\ast))^\circ = \emptyset$. We conclude with Proposition \ref{L-coverings} and 
Corollary \ref{one-parameter classifications} (ii) that $\mathscr{K}(k) = \mathscr{K}_\ell(k) = \emptyset$.
\\[1ex]
{\it (ii)} Lemma \ref{norm and squares} (i) implies that $|(k^\ast/{\rm im}(n_{\ell/k}^\ast))^\circ| = |(k^\ast/k_{sq}^\ast)^\circ| = 1$. Hence every $c_3 \in k^\ast \setminus {\rm im}(n_{\ell/k}^\ast)$ 
provides a transversal $S_\ell = \{c_3\}$ for $(k^\ast/{\rm im}(n_{\ell/k}^\ast))^\circ$. We conclude with Corollary \ref{one-parameter classifications} (i) and Proposition \ref{L-coverings} that
$ \{ A(\ell, (1,0,c_3)) \}$ classifies $\mathscr{S}_\ell(k) = \mathscr{S}(k)$.
\\[1ex] 
{\it (iii)} By hypothesis, $(k^\ast/{\rm im}(n_{\ell/k}^\ast))^\circ = \emptyset$. We conclude with Proposition \ref{L-coverings} and Corollary \ref{one-parameter classifications} (i) that
$\mathscr{S}(k) = \mathscr{S}_\ell(k) = \emptyset$. With (i) and Proposition \ref{decomposition}, the identities $\mathscr{C}(k) = \mathscr{N}(k)$ and $C_\ell = C_\ell^\mathscr{N}$ follow. Assume that $c_2 = 0$ 
for some $\underline{c} \in C_\ell$. Then $\underline{c} = (c_1, 0, c_3) \in k^3$, such that the associated function
\[ q_{\underline{c}}: \ell^2 \to \ell,\ q_{\underline{c}}(x,y) = (1-c_1)x^2 + c_1 x\overline{x} - c_3 y\overline{y} \]
is anisotropic. So $c_3 = -q_{\underline{c}}(0,1) \in k^\ast$. As $n_{\ell/k}^\ast$ is surjective, there exists an element $\eta \in \ell^\ast$ such that $\eta\overline{\eta} = c_3^{-1}$. Then
$q_{\underline{c}}(1,\eta) = 0$,\vspace{0,1cm} which contradicts the anisotropy of $q_{\underline{c}}$. Thus $c_2 \not= 0$ for all $\underline{c} \in C_\ell$.
\end{proof}

\noindent
Proposition \ref{dichotomy}, combined with Proposition \ref{L-coverings} and Theorem \ref{reduction of classifications} (ii), yields the following corollary.

\begin{cor} \label{uqec classification}
Let $k$ be a field of characteristic not 2 that has unique quadratic extension class, represented by $\ell$, and let $T_\ell^{\mathscr{N}}$ be a transversal for $C_\ell^{\mathscr{N}}/\ell^\ast$. 
\\[1ex]
(i) If $n_{\ell/k}^\ast$ is not surjective and $c_3 \in k^\ast \setminus {\rm im}(n_{\ell/k}^\ast)$, then 
\[ \mathscr{F}_\ell(T_\ell^{\mathscr{N}}) \sqcup \{ A(\ell,(1,0,c_3)) \}\ \mbox{classifies}\ \mathscr{C}(k) = \mathscr{N}(k) \amalg \mathscr{S}(k). \]
(ii) If $n_{\ell/k}^\ast$ is surjective, then
\[ \mathscr{F}_\ell(T_\ell^{\mathscr{N}})\ \mbox{classifies}\ \mathscr{C}(k) = \mathscr{N}(k). \]
\end{cor}

\noindent
For every field $k$ of characteristic not 2 that has unique quadratic extension class, represented by $\ell$, Corollary \ref{uqec classification} reduces the problem of classifying $\mathscr{C}(k)$ to 
the following three problems.
\\[1ex]
{\bf Problem 1$'$.} {\it Display the subset $C_\ell^\mathscr{N} \subset k^3$ explicitly.} 
\\[1ex]
{\bf Problem 2$'$.} {\it Find a transversal $T_\ell^\mathscr{N}$ for $C_\ell^\mathscr{N}/\ell^\ast$.} 
\\[1ex]
{\bf Problem 3$'$.} {\it Determine whether the group morphism $n_{\ell/k}^\ast: \ell^\ast \to k^\ast$ is surjective or not. If $n_{\ell/k}^\ast$ is not surjective, choose an element 
$c_3 \in k^\ast \setminus {\rm im}(n_{\ell/k}^\ast)$.} 
\\[1ex]
The preceding results reveal a {\it dichotomy of types}, regarding fields $k$ that have characteristic not 2 and unique quadratic extension class. The {\it first type} is characterized by 
non-surjectivity of the norm morphism $n_{\ell/k}^\ast$, or equivalently, by non-emptiness of the groupoid $\mathscr{S}(k)$, whereas the {\it second 
type} is characterized by surjectivity of $n_{\ell/k}^\ast$, or equivalently, by emptiness of $\mathscr{S}(k)$.

In the remainder of the present section we classify the groupoid $\mathscr{C}(k)$ for two classes of fields $k$, having characteristic not 2 and unique quadratic extension class. The first class is formed by 
all square-ordered fields (Subsection 5.3), and all of these are of the first type. The second class is formed by all finite fields of odd order (Subsection 5.4), and all of these are of the second type.

\subsection{Classification of $\mathscr{C}(k)$ for square-ordered fields $k$}

Recall our notation $k_{sq} = \{ x^2\ |\ x \in k \}$, valid for every field $k$. If $k$ is an ordered field, then $k$ has characteristic $0$, the set $k_{\ge 0} = \{ x \in k\ |\ x \ge 0 \}$ of all 
non-negative elements in $k$ is defined, and the inclusion $k_{sq} \subset k_{\ge 0}$ holds. By a {\it square-ordered field} we mean an ordered field $k$ such that $k_{sq} = k_{\ge 0}$. The real number 
field $\R$ and, more generally, all real closed fields $R$ are examples of square-ordered fields \cite[VI, Theorem 2.3]{Gr07}. 

In this subsection, we classify $\mathscr{C}(k)$ for all square-ordered fields $k$ (Corollary \ref{square-ordered classification}), by way of establishing that $k$ has unique quadratic extension class, 
solving Problems $1'$-$3'$, and applying Corollary \ref{uqec classification}. 

For technical reasons we introduce the partition $C_\ell^{\mathscr{N}} = C_\ell^{\mathscr{N}_0} \sqcup C_\ell^{\mathscr{N}_1}$, where $C_\ell^{\mathscr{N}_0} = \{ \underline{c} \in C_\ell^{\mathscr{N}}\ |\ c_2 = 0 \}$ and 
$C_\ell^{\mathscr{N}_1} = \{ \underline{c} \in C_\ell^{\mathscr{N}}\ |\ c_2 \not= 0 \}$. The set of all $\ell^\ast$-equivalence classes of $C_\ell^{\mathscr{N}_0}$ and $C_\ell^{\mathscr{N}_1}$ is denoted by 
$C_\ell^{\mathscr{N}_0}/\ell^\ast$ and $C_\ell^{\mathscr{N}_1}/\ell^\ast$, respectively.

\begin{pro} \label{square-ordered fields}
Let $k$ be a square-ordered field.
\\[1ex]
(i) Then $k$ has unique quadratic extension class, represented by $\ell = k\left(\sqrt{-1}~\right)$.
\\[1ex]
(ii) The map $n_{\ell/k}^\ast: \ell^\ast \to k^\ast$ is not surjective, and $-1 \in k^\ast \setminus 
{\rm im}(n_{\ell/k}^\ast)$.
\\[1ex]
(iii) The subset $C_\ell \subset k^3$ admits the explicit display
\[ C_\ell = \left\{ (c_1, c_2, c_3) \in k^3\ \left|\ c_1 > \frac{1}{2}\ \wedge\ c_3 < c_2 < -c_3 \right.
\right\}. \] 
(iv) The subsets $C_\ell^{\mathscr{N}_0} \subset C_\ell^{\mathscr{N}}$ and $C_\ell^{\mathscr{N}_1} \subset C_\ell^{\mathscr{N}}$ admit the explicit displays
\begin{eqnarray*} 
C_\ell^{\mathscr{N}_0} & = & \left\{ (c_1, 0, c_3) \in k^3\ \left|\ c_1 > \frac{1}{2}\ \wedge\ c_3 < 0\ 
\wedge\ c_1 \not= 1 \right. \right\},\ \mbox{and} 
\\[1ex]
C_\ell^{\mathscr{N}_1} & = & \left\{ (c_1, c_2, c_3) \in k^3\ \left|\ c_1 > \frac{1}{2}\ \wedge\ c_3 < c_2 < 
-c_3\ \wedge\ c_2 \not= 0 \right. \right\}.
\end{eqnarray*}  
(v) The subsets
\begin{eqnarray*} 
T_\ell^{\mathscr{N}_0} & = & \left\{  (c_1, 0, -1) \in k^3\ \left|\ c_1 > \frac{1}{2}\ \wedge\ c_1 \not= 1 \right. \right\}\ \subset\ C_\ell^{\mathscr{N}_0}\ \mbox{and}
\\[1ex]
T_\ell^{\mathscr{N}_1} & = & \left\{  (c_1, 1, c_3) \in k^3\ \left|\ c_1 > \frac{1}{2}\ \wedge\ c_3 < -1 \right. \right\}\ \subset\ C_\ell^{\mathscr{N}_1}
\end{eqnarray*}  
are transversals for $C_\ell^{\mathscr{N}_0}/\ell^\ast$ and $C_\ell^{\mathscr{N}_1}/\ell^\ast$, respectively.
\\[1ex]
(vi) The set $T_\ell^{\mathscr{N}} = T_\ell^{\mathscr{N}_0} \sqcup T_\ell^{\mathscr{N}_1}$ is a transversal for $C_\ell^{\mathscr{N}}/\ell^\ast$.
\end{pro}

\begin{proof}
{\it (i)} Being a square-ordered field, $k$ satisfies the identities $k_{sq}^\ast = k_{>0}$ and $k^\ast \setminus 
k_{sq}^\ast = k_{<0}$, whence $-1 \in k^\ast \setminus k_{sq}^\ast$ and $|k^\ast/k_{sq}^\ast| = 2$. Statement (i) now follows from Proposition \ref{quadratic extensions}, (i) and (v).
\\[1ex]
{\it (ii)} General elements $a \in \ell^\ast$ are of the form $a = x+iy$, where $(x,y) \in k^2 \setminus \{(0,0)\}$ and $i = \sqrt{-1}$. Thus $n_{\ell/k}^\ast(a) = a\overline{a} = x^2+y^2 \in k_{>0}$ shows 
that $n_{\ell/k}^\ast$ is not surjective, and $-1 \in k^\ast \setminus {\rm im}(n_{\ell/k}^\ast)$. 
\\[1ex]
{\it (iii)} Set $C = \left\{ \underline{c} \in k^3\ \left|\ c_1 > \frac{1}{2}\ \wedge\ c_3 < c_2 < -c_3 \right. \right\}$.\vspace{0,1cm} We aim to show that $C_\ell = C$. We refer to Subsections 2.2 and 3.1.

Let $\underline{c} \in C_\ell$. For all $(x,y) \in k^2$ we have $(\overline{x}, \overline{y}) = (x,y)$, and hence 
\[ q_{\underline{c}}(x,y) = x^2 - (c_2 + c_3)y^2. \]
If $c_2 + c_3 \ge 0$, then $c_2 + c_3 = \xi^2$ for some $\xi \in k$, whence $q_{\underline{c}}(\xi, 1) = 0$, contradicting $\underline{c} \in C_\ell$. So $c_2 + c_3 < 0$.

For all $(x,y) \in k \times {\rm Im}(\ell)$ we have $(\overline{x},\overline{y}) = (x,-y)$, and hence 
\[ q_{\underline{c}}(x,y) = x^2 - (c_2 - c_3)y^2. \]
If $-c_2+c_3 \ge 0$, then $-c_2 + c_3 = \xi^2$ for some $\xi \in k$, whence $q_{\underline{c}}(\xi, i) = 0$, contradicting $\underline{c} \in C_\ell$. So $-c_2 + c_3 < 0$.

For all $(x,y) \in {\rm Im}(\ell) \times k$ we have $(\overline{x},\overline{y}) = (-x,y)$, and hence 
\[ q_{\underline{c}}(x,y) = (1-2c_1)x^2 - (c_2 + c_3)y^2. \]
If $1-2c_1 \ge 0$, then $\frac{1-2c_1}{-(c_2 + c_3)} \ge 0$ and hence $\frac{1-2c_1}{-(c_2 + c_3)} = \xi^2$\vspace{1ex} for some $\xi \in k$, whence $q_{\underline{c}}(i, \xi) = 0$, contradicting 
$\underline{c} \in C_\ell$. So $1-2c_1 < 0$. Summarizing, we have shown that $\underline{c} \in C$.

Conversely, each $\underline{c} \in C$ determines a pair of quadratic forms
\mbox{$q_{\underline{c}}^\mu: \ell^2 \to k,$}\ $\mu \in \underline{2}$, defined by
\[ q_{\underline{c}}(x,y) = q_{\underline{c}}^1(x,y) + iq_{\underline{c}}^2(x,y). \]
Writing $(x,y) = (x_1 + ix_2, y_1 + iy_2)$ with $(x_1, x_2, y_1, y_2) \in k^4$, we obtain the identity 
\[  q_{\underline{c}}^1(x,y)  =  x_1^2 - (1-2c_1)x_2^2 - (c_2 + c_3)y_1^2 + (c_2 - c_3)y_2^2, \]
where all of the coefficients $1, -(1-2c_1), -(c_2 + c_3)$ and $c_2 - c_3$ are positive elements in $k$, by hypothesis. Thus, if $q_{\underline{c}}(x,y) = 0$, then $q_{\underline{c}}^1(x,y) = 0$ implies 
$(x,y) = (0,0)$. Accordingly, $\underline{c} \in C_\ell$.
\\[1ex]
{\it (iv)} This statement is a direct consequence of (iii), Proposition \ref{types of admissible triples} (i),\vspace{0,1cm} and the definition of the partition 
$C_\ell^{\mathscr{N}} = C_\ell^{\mathscr{N}_0} \sqcup C_\ell^{\mathscr{N}_1}$.
\\[1ex]
{\it (v)} In view of the definition of $\ell^\ast$-equivalence (Subsections 5.1 and 3.3), and because every $c_3 < 0$ can be written
$c_3 = -a\overline{a}$ for some $a \in \ell^\ast$, and every $c_2 \not= 0$ can be written $c_2 = a^2$ for some $a \in \ell^\ast$, statement (v) is an easy consequence of (iv). 
\\[1ex]
{\it (vi)} Since $C_\ell^{\mathscr{N}}/\ell^\ast = C_\ell^{\mathscr{N}_0}/\ell^\ast \sqcup C_\ell^{\mathscr{N}_1}/\ell^\ast$, (vi) follows directly from (v).
\end{proof}

\begin{cor} \label{square-ordered classification}
Let $k$ be a square-ordered field, let $\ell = k\left(\sqrt{-1}~\right)$, and let $T_\ell^{\mathscr{N}} \subset C_\ell^{\mathscr{N}}$ be the subset displayed in Proposition \ref{square-ordered fields} (v)-(vi). Then
\[ \mathscr{F}_\ell(T_\ell^{\mathscr{N}}) \sqcup \{ A(\ell,(1,0,-1)) \}\ \mbox{classifies}\ \mathscr{C}(k) = \mathscr{N}(k) \amalg \mathscr{S}(k). \]
\end{cor}

\begin{proof}
Every ordered field $k$ has characteristic $0$. By Proposition \ref{square-ordered fields}, every square-ordered field $k$ has unique quadratic extension class, represented \mbox{by $\ell$,} the subset 
$T_\ell^{\mathscr{N}} \subset C_\ell^{\mathscr{N}}$ is a transversal for $C_\ell^{\mathscr{N}}/\ell^\ast$, and $-1 \in k^\ast \setminus {\rm im}(n_{\ell/k}^\ast)$. The classification statement is now a special case of 
Corollary \ref{uqec classification} (i).
\end{proof}

\subsection{Classification of $\mathscr{C}(k)$ for finite fields $k$ of odd order}

Throughout this subsection, $k$ is a field of the second type, i.e.~$k$ has cha\-racteristic not 2 and unique quadratic extension class, represented by $\ell$, such that the norm morphism 
$n_{\ell/k}^\ast: \ell^\ast \to k^\ast$ is surjective. All finite fields $\F_q$ of odd order $q$ are fields of the second type. 

We study the problem of classifying $\mathscr{C}(k)$, which, by Corollary \ref{uqec classification} (ii), is reduced to Problems $1'$ and $2'$. The approach to Problems $1'$ and $2'$ which we present in the
following discussion is a generalized and streamlined version of \cite[p.~218-219]{BA15}, where the ground field is assumed to be $\F_q$, with $q$ odd.  

Let $k$ be a field of the second type, and let $k \subset \ell$ be a quadratic extension. We introduce the map
\[ |\cdot|: k \to k,\ |x| = \left\{ \begin{array}{rcl} x & \mbox{if} & x \in k_{sq} \\ -x & \mbox{if} & x \not\in k_{sq}, \end{array} \right. \]
and we note that $|x| = \sqrt{x}\ \overline{\sqrt{x}} = n_{\ell/k}(\sqrt{x})$ for all $x \in k$. Also,\vspace{0,1cm} $|\cdot|$ induces a group endomorphism $|\cdot|^\ast: k^\ast \to k^\ast$, which is 
an automorphism if and only if $-1 \in k_{sq}^\ast$.
With any pair of scalars $(a,b) \in k^2$ we associate a function 
\[ h_{a,b}: \ell^2 \to \ell,\ h_{a,b}(x,y) = x^2 - y^2 + ax\overline{x} + by\overline{y}. \] 
We say that a pair $(a,b) \in k^2$ is {\it $\ell$-admissible} if the function $h_{a,b}: \ell^2 \to \ell$ is anisotropic. The $\ell$-admissible pairs in $k^2$ form a subset 
\[ B_\ell = \left\{ (a,b) \in k^2\ \left|\ h_{a,b}^{-1}(0) = \{(0,0)\} \right. \right\} \]
of $k^2$. Since $k$ has unique quadratic extension class, the subsets $C_\ell^{{\mathscr N}} \subset k^3$ and $B_\ell \subset k^2$ are in fact uniquely determined by $k$. 

The following lemma reduces Problem $1'$ to the problem of displaying the subset $B_\ell \subset k^2$ explicitly.

\begin{lem}\label{from C to B}
Let $k$ be a field of the second type, and let $\underline{c} \in k^3$. Then $\underline{c} \in C_\ell^{\mathscr N}$ if and only if $(1-c_1)c_2 \not= 0$ and
\[ (a,b) = \left( \frac{c_1}{|1-c_1|}, \frac{-c_3}{|c_2|} \right) \in B_\ell. \]
\end{lem}

\begin{proof}
Let $\underline{c} \in C_\ell^{\mathscr N}$. Then $c_2 \not= 0$ holds by Proposition \ref{dichotomy} (iii). Since $n_{\ell/k}^\ast$ is surjective, there exists a $\xi \in \ell$ such that 
$\xi\overline{\xi} = c_2+c_3$. If $1-c_1 = 0$, then $q_{\underline{c}}(\xi,1) = 0$ contradicts the anisotropy of $q_{\underline{c}}$. Hence $1-c_1 \not= 0$.\vspace{0,1cm} By Lemma \ref{norm and squares} (ii) 
there exist $\alpha, \beta \in \ell^\ast$ such that $(\alpha^2,\beta^2) = (1-c_1, c_2)$.\vspace{0,1cm} Then $q_{\underline{c}}(x,y) = h_{a,b}(\alpha x, \beta y)$ holds for all $(x,y) \in \ell^2$. 
Since $q_{\underline{c}}$ is anisotropic, it follows that so is $h_{a,b}$. Thus $(a,b) \in B_\ell$. 

Conversely, let $\underline{c} \in k^3$ be such that $(1-c_1)c_2 \not= 0$ and $(a,b) \in B_\ell$. Rever\-sing the last argument in the above reasoning we find that $q_{\underline{c}}$ is anisotropic, 
i.e.~$\underline{c} \in  C_\ell$, and $C_\ell = C_\ell^{\mathscr N}$ holds by Proposition \ref{dichotomy} (iii).
\end{proof}

\noindent
The next proposition aims at an explicit display of the subset $B_\ell \subset k^2$. The proof presented here generalizes a line of arguments, contained in \cite[Proof of Proposition 1]{BA15},
from finite fields $\F_q$ to fields $k$ of the second type.

\begin{pro} \label{B}
Let $k$ be a field of the second type, with associated subsets $M_1 = \{ m \in k\ |\ m^2 - 1 \in k_{sq} \}$ and $M_2 = \{ m \in k\ |\ m^2 - 1 \not\in k_{sq}^\ast \}$ of $k$. Then, for all $(a,b) \in k^2$, the 
following holds true.
\\[1ex]
(i) If $b = -a$, then $(a,b) \not\in B_\ell$.
\\[1ex]
(ii) If $b = a$, then $(a,b) \in B_\ell$ if and only if $1-a^2 \not\in k_{sq}$.
\\[1ex]
(iii) If $b^2 \not= a^2$, then $(a,b) \in B_\ell$ if and only if 
\[ (a+b)m_1 \not= 2 + (a-b)m_2\hspace{0,5cm} \forall (m_1,m_2) \in M_1 \times M_2. \]
\end{pro}

\begin{proof}
{\it (i)} The identity $h_{a,-a}(1,1) = 0$ shows that $h_{a,-a}$ is isotropic.
\\[1ex]
For the remainder of the proof, let $(a,b) \in k^2$ such that $b \not= -a$. For all $(x,y) \in \ell^2$ such that $x^2 - y^2 \in k^\ast$, we set $t = x^2 - y^2$ and $z = x + y$. 
Note that $(t,z) \in k^\ast \times \ell^\ast$. With this notation, the following statement (*) is proved in \cite[Proof of Proposition 1]{BA15}.
\\[1ex]
(*) {\it Let $(a,b) \in k^2$ with $b \not= -a$. Then $(a,b) \in B_\ell$ if and only if 
\[ \frac{a+b}{4z\overline{z}} t^2 + \left( 1 + \frac{a-b}{4}\left( \frac{z}{\overline{z}} + 
\frac{\overline{z}}{z}  \right) \right) t + \frac{a+b}{4} z\overline{z}\ \not=\ 0 \] 
for all $(x,y) \in \ell^2$ with $x^2 - y^2 \in k^\ast$.}
\\[1ex]
For every pair $(t,z) \in k^\ast \times \ell^\ast$ there is\vspace{1ex} a pair $(x,y) \in \ell^2$ such that $x^2 - y^2 = t$ and $x+y = z$. Indeed, 
$(x,y) = \left( \frac{z^2 + t}{2z}, \frac{z^2 - t}{2z} \right)$\vspace{1ex} will do. For each $z \in \ell^\ast$, the quadratic polynomial in $t$ appearing as left hand side in (*) has all of its 
coefficients in $k$, and its discriminant is 
\[ \Delta(z) = \left( 1 + \frac{a-b}{4}\left( \frac{z}{\overline{z}} + 
\frac{\overline{z}}{z}  \right) \right)^2 - \frac{(a+b)^2}{4}\ . \]
Accordingly, (*) is equivalent to the following statement (**). 
\\[1ex]
(**) {\it Let $(a,b) \in k^2$ with $b \not= -a$. Then $(a,b) \in B_\ell$ if and only if $\Delta(z) \not\in k_{sq}$ for all $z \in \ell^\ast$.}
\\[1ex]
We proceed to prove (ii) and (iii) by use of (**).
\\[2ex]
{\it (ii)} Let $(a,b) = (a,a)$ with $a \in k^\ast$. Then $\Delta(z) = 1 - a^2$ for all $z \in 
\ell^\ast$. Thus, (**) asserts that $(a,a) \in B_\ell$ if and only if $1-a^2 \not\in k_{sq}$.
\\[2ex]
{\it (iii)} Let $(a,b) \in k^2$ with $a^2 \not= b^2$. Following \cite{BA15}, we set $s = \frac{a+b}{2}$,
\[ M_s = \left\{ m \in k\ \left|\ m^2 - s^2 \in k_{sq} \right.\right\}, \mbox{and}\ 
N = \left\{ \left. 1 + \frac{a-b}{4} \left( \frac{z}{\overline{z}} + \frac{\overline{z}}{z} \right)\ \right|\ z \in \ell^\ast \right\}. \] 
With this notation, $\{ \Delta(z)\ |\ z \in \ell^\ast \} = \{ n^2-s^2\ |\ n \in N \}$. Hence (**) states that $(a,b) \in B_\ell$ if and only if 
$M_s \cap N = \emptyset$. Moreover, $M_s = sM_1$ and $N = 1 + \frac{a-b}{2}M_2$. To justify the latter identity, we observe with Hilbert's Theorem 90 \cite[V, Lemma 7.7]{Gr07} that
\[ \left\{ \frac{1}{2} \left( \left. \frac{z}{\overline{z}} + \frac{\overline{z}}{z} \right)\ \right|\ z \in \ell^\ast \right\} = \left\{ \left. \frac{1}{2}(w + \overline{w})\ \right|\ w \in \S \right\}, \] 
where $\S = {\rm ker}(n_{\ell/k}^\ast)$. Choose $i \in {\rm Im}(\ell) \setminus \{0\}$. Then $i^2 = -t \in k^\ast \setminus k_{sq}^\ast$ and 
$\S = \left\{ w_1 + iw_2\ \left|\ w_1, w_2 \in k\ \wedge\ w_1^2 + tw_2^2 = 1 \right.\right\}$ and $|k^\ast/k_{sq}^\ast| = 2$ implies  
\[ \left\{ \left. \frac{1}{2}(w + \overline{w})\ \right|\ w \in \S \right\} = \left\{ w_1 \in k\ \left|\ w_1^2 - 1 \not\in k_{sq}^\ast \right.\right\} = M_2, \] 
whence\vspace{0,1cm} the claim follows. The above identities for $M_s$ and $N$ show that $M_s \cap N = \emptyset$ if and only if
$\frac{a+b}{2} m_1 \not= 1 + \frac{a-b}{2}m_2$ for all $(m_1, m_2) \in M_1 \times M_2$.
\end{proof}

\noindent
If $k$ is a finite field of odd order, then the above results can be supplemented by a counting argument, that readily yields solutions to Problems $1'$ and $2'$, and thereby a classification of 
$\mathscr{C}(k)$. In statements (i) and (ii) of Proposition \ref{B for finite fields} below we recover \cite[Proposition 1]{BA15} and \cite[Main Theorem]{BA15}. The latter statement solves Problem $1'$.
The counting argument used in the proof of Proposition \ref{B for finite fields} (i) is reproduced from \cite[Proof of Proposition 1]{BA15}, to keep our presentation self-contained. In contrast, 
the transversal $T_\ell^\mathscr{N}$ that solves Problem $2'$ (Proposition \ref{B for finite fields} (iii)) and the corresponding classification of $\mathscr{C}(k)$ (Corollary \ref{finite classification}) 
are beyond the scope of \cite{BA15}.

\begin{pro} \label{B for finite fields}
Let $k$ be a finite field of odd order. Then 
\\[1ex]
(i) $B_\ell = \left\{ (a,b) \in k^2\ \left|\ a=b\ \wedge\ 1-a^2 \not\in k_{sq} \right.\right\}$, 
\\[1ex]
(ii) $C_\ell^\mathscr{N} = \left\{ \left. \left( c_1, c_2, \frac{-c_1|c_2|}{|1-c_1|} \right) \in k^3\ \right|\ (1-c_1)c_2 \not= 0\ \wedge\ 1-2c_1 \not\in k_{sq}\ \right\}$, and  
\\[1ex]
(iii) $T_\ell^\mathscr{N} = \left\{ \left. \left(c_1, 1, \frac{-c_1}{|1-c_1|} \right) \in k^3\ \right|\ c_1 \not= 1\ \wedge\ 1-2c_1 \not\in k_{sq}\ \right\}$ is a transversal for $C_\ell^\mathscr{N}/\ell^\ast$.
\end{pro}

\begin{proof}
{\it (i)} As every finite field $k$ of odd order is of the second type, Proposition \ref{B} applies. Therefore it suffices to show that, for all $(a,b) \in k^2$ with $a^2 \not= b^2$, the identity
$(a+b)m_1 = 2 + (a-b)m_2$ has a solution $(m_1,m_2) \in M_1 \times M_2$. 

Indeed, every $(a,b) \in k^2$ with $a^2 \not= b^2$ determines a bijective map 
\[ \varphi: k \to k,\ \varphi(x) = \frac{2}{a+b} + \frac{a-b}{a+b}x. \] 
Setting $q = |k|$, we have that $|M_1| = \frac{q+1}{2}$ and $|M_2| = \frac{q+3}{2}$, and hence
\begin{eqnarray*}
|M_1 \cap \varphi(M_2)| & = & |M_1| + |\varphi(M_2)| - |M_1 \cup \varphi(M_2)| \\[1ex]
 & \ge & |M_1| + |M_2| - |k| \\[1ex]
 & = & \frac{q+1}{2} + \frac{q+3}{2} - q\ =\ 2.  
\end{eqnarray*}
Thus, there exists a pair $(m_1,m_2) \in M_1 \times M_2$ such that $m_1 = \varphi (m_2)$. 
\\[1ex]
{\it (ii)} Applying Lemma \ref{from C to B} and (i), we obtain for all $\underline{c} \in k^3$ the chain of equivalences
\begin{eqnarray*}
\underline{c} \in C_\ell^{\mathscr{N}} & \Leftrightarrow & (1-c_1)c_2 \not= 0\ \wedge\  \frac{c_1}{|1-c_1|} = \frac{-c_3}{|c_2|}\ \wedge\ 1-2c_1 \not\in k_{sq} \\
 &  \Leftrightarrow & (1-c_1)c_2 \not= 0\ \wedge\ 1-2c_1 \not\in k_{sq}\ \wedge\ c_3 =  \frac{-c_1|c_2|}{|1-c_1|}.
\end{eqnarray*} 
{\it (iii)} By definition of $\ell^\ast$-equivalence, and because every $c_2 \in k^\ast$ can be written $c_2 = a^2$ for some $a \in \ell^\ast$ (Lemma \ref{norm and squares} (ii)), statement (iii) 
is a straightforward consequence of (ii).
\end{proof}

\begin{cor} \label{finite classification}
Let $k$ be a finite field of odd order, let $k \subset \ell$ be a quadratic extension, and let $T_\ell^{\mathscr{N}}$ be the transversal displayed in Proposition \ref{B for finite fields} (iii). Then
$\mathscr{F}_\ell(T_\ell^{\mathscr{N}})\ \mbox{classifies}\ \mathscr{C}(k) = \mathscr{N}(k)$. 
\end{cor}

\begin{proof}
Corollary \ref{uqec classification} (ii) applies, and proves the statement.
\end{proof}

\section{Kleinian coverings by group actions}

In this last section, we investigate the structure of the groupoid $\mathscr{C}(k)$, still assuming that ${\rm char}(k) \not= 2$. In order to reach this goal, we first present the concept of a
covering of a groupoid by group actions, in great generality (Subsection 6.1). Next, we deepen our knowledge of the particular groupoid $\mathscr{C}(k)$ by gaining more insight into its morphism sets 
(Subsection 6.2). In the final Subsection 6.3, we mould the accumulated information on $\mathscr{C}(k)$ into the concept provided in Subsection 6.1. As an application, we determine the automorphism groups
of all division algebras in $\mathscr{C}(k)$.

\subsection{The concept of a covering of a groupoid by group actions}

The concept to be presented here refines and extends the concept of a description of a groupoid by a group action, introduced in \cite[Section 5]{Di12}.

Every left action $G \times X \to X,\ (g,x) \mapsto gx$ of a group $G$ on a set $X$ gives rise to a small groupoid 
$_GX$, with object set $X$ and morphism sets 
\[ {\rm Mor}_{_{_GX}}(x,y) = \{ (g,x,y) \in G \times \{x\} \times \{y\}\ |\ gx = y \}. \]
We say that a groupoid $\mathscr{X}$ is a {\it group action groupoid} if $\mathscr{X} =\ _GX$ for some left group action $G \times X \to X$.

For any groupoid $\mathscr{Y}$, a {\it description of} $\mathscr{Y}$ {\it by a group action} (or briefly, a {\it description} of $\mathscr{Y}$) is, by definition, a quadruple $(X,G,\gamma,\mathscr{F})$, 
composed of a set $X$, a group $G$, a group action $\gamma: G \times X \to X$, and a functor $\mathscr{F}:\ _GX \to \mathscr{Y}$ which is dense, faithful, and {\it quasi-full} in the sense that 
\[ {\rm Mor}_{\mathscr{Y}}(\mathscr{F}(x),\mathscr{F}(y)) \not= \emptyset\ \Rightarrow\ {\rm Mor}_{_{_GX}}(x,y)  \not= \emptyset \]
holds for all $(x,y) \in X \times X$. A description $(X,G,\gamma,\mathscr{F})$ of a groupoid $\mathscr{Y}$ is said to be {\it full} if the functor $\mathscr{F}$ is full, i.e.~if $\mathscr{F}$ is an 
equivalence of categories.\footnote{Full descriptions in this sense were simply called descriptions in \cite[Section 5]{Di12}, while descriptions in our present sense were not considered in 
\cite[Section 5]{Di12}.} 
  
If $(X,G,\gamma,\mathscr{F})$ is a description of a groupoid $\mathscr{Y}$ and $T \subset X$ is a transversal for the orbit set $X/\gamma$ of the group action $\gamma$, then the subset 
$\mathscr{F}(T) \subset \mathscr{Y}$ is a transversal for the set $\mathscr{Y}/\hspace{-0,1cm}\simeq$ of all isoclasses of $\mathscr{Y}$. Accordingly, $\mathscr{Y}$ is svelte (Subsection 2.1), and 
$\mathscr{F}(T)$ classifies $\mathscr{Y}$. 

Now, let $(\mathscr{Y}_i)_{i \in I}$ be a family of full subgroupoids of a groupoid $\mathscr{Y}$, and suppose that a description $(X_i,G_i,\gamma_i,\mathscr{F}_i)$ of $\mathscr{Y}_i$ is given for each 
$i \in I$. We say that the family of descriptions $(X_i,G_i,\gamma_i,\mathscr{F}_i)_{i \in I}$ is a {\it covering of} $\mathscr{Y}$ {\it by group actions} if the identity of object classes 
$\mathscr{Y} = \bigcup_{i \in I} \mathscr{Y}_i$ holds true.

\subsection{Morphism sets in $\mathscr{C}(k)$}

Throughout this subsection, $k \subset \ell$ is a quadratic extension in characteristic not 2, and $\underline{c},\underline{d} \in C_\ell$ are $\ell$-admissible triples. Our study of the morphism sets in 
$\mathscr{C}(k)$ is based on a characterization of the special morphisms $\varphi_a$ and $\psi_a$ from $A(\ell,\underline{c})$ to $A(\ell,\underline{d})$, constructed in Subsection 3.3, among 
all morphisms from $A(\ell,\underline{c})$ to $A(\ell,\underline{d})$. Regarding the notation $1_A\ell$, see Proposition \ref{constructed algebras} (ii).

\begin{pro} \label{constructed morphisms vs all morphisms}
For every quadratic extension $k \subset \ell$ in characteristic \mbox{not 2,} and for all $\underline{c},\underline{d} \in C_\ell$, the following holds true. 
\\[1ex]
(i) $\left\{ \varphi_a,\psi_a \left|\ a \in \ell^\ast(\underline{c},\underline{d}) \right. \right\} = \left\{ \mu \in {\rm Mor}_{\mathscr{C}(k)}(A(\ell,\underline{c}),A(\ell,\underline{d})) \left|\ 
\mu(1_A\ell) = 1_A\ell \right. \right\}$.
\\[1ex]
(ii) If $(\underline{c},\underline{d}) \in \left( C_\ell^\mathscr{N} \times C_\ell^\mathscr{N} \right) \sqcup \left( C_\ell^\mathscr{K} \times C_\ell^\mathscr{K} \right)$, then $\mu(1_A\ell) = 1_A\ell$ 
holds\vspace{0,1cm} for all morphisms $\mu \in {\rm Mor}_{\mathscr{C}(k)}(A(\ell,\underline{c}),A(\ell,\underline{d}))$.
\end{pro}

\begin{proof}
{\it (i)} The inclusion ``$\subset$'' holds by definition of $\varphi_a$ and $\psi_a$. Conversely, let $\mu: A(\ell,\underline{c}) \to A(\ell,\underline{d})$ be a morphism in $\mathscr{C}(k)$, such that
$\mu(1_A\ell) = 1_A\ell$. Then $\mu$ induces a Galois automorphism $\mu_\ell \in {\rm Gal}(\ell/k)$, and ${\rm Gal}(\ell/k) = \langle \sigma \rangle$ has order 2. If $\mu_\ell = \I_\ell$, then the arguments 
contained in the proof of Theorem \ref{reduction of morphisms}, Step 3 and Step 4, show that $\mu = \varphi_a$ for some $a \in \ell^\ast(\underline{c},\underline{d})$. If $\mu_\ell = \sigma$, then the same 
reasoning applies to $\mu\psi_1$, where $\psi_1 \in {\rm Aut}(A(\ell,\underline{c}))$, because $(\mu\psi_1)_\ell = \I_\ell$. Thus $\mu\psi_1 = \varphi_a$ and hence $\mu = \varphi_a\psi_1 = \psi_a$, for 
some $a \in \ell^\ast(\underline{c},\underline{d})$. 
\\[1ex]
{\it (ii)} See proof of Theorem \ref{reduction of morphisms}, Step 1.
\end{proof}

\begin{cor} \label{when all morphisms are constructed}
If $(\underline{c},\underline{d}) \in \left( C_\ell^\mathscr{N} \times C_\ell^\mathscr{N} \right) \sqcup \left( C_\ell^\mathscr{K} \times C_\ell^\mathscr{K} \right)$, then 
\[ \left\{ \varphi_a,\psi_a \left|\ a \in \ell^\ast(\underline{c},\underline{d}) \right. \right\} = {\rm Mor}_{\mathscr{C}(k)}(A(\ell,\underline{c}),A(\ell,\underline{d})). \]
\end{cor}

\vspace{0,1cm}
\noindent
Recall from Proposition \ref{constructed algebras} (iii)-(iv) and Lemma \ref{constructed morphisms} (iii) that, for all $\underline{c} \in C_\ell$, the division algebra $A(\ell,\underline{c})$ admits the 
Kleinian pair $(\alpha,\beta) = (\varphi_{-1},\psi_1)$. 

\begin{pro} \label{Kleinian subgroup}
For every quadratic extension $k \subset \ell$ in characteristic \mbox{not 2,} and for all $\underline{c} \in C_\ell$, the following holds true. 
\\[1ex]
(i) If $c_2 = 0$, then the subgroup $\langle \varphi_{-1},\psi_1 \rangle < {\rm Aut}(A(\ell,\underline{c}))$ is proper.
\\[1ex]
(ii) If $c_2 \not= 0$, then $\langle \varphi_{-1},\psi_1 \rangle = {\rm Aut}(A(\ell,\underline{c}))$.
\end{pro}

\begin{proof}
For definition of the subgroup $\S < \ell^\ast$, see Subsection 2.2. 
\\[1ex]
{\it (i)} If $c_2 = 0$, then $c_3 \not= 0$, whence $\ell^\ast(\underline{c},\underline{c}) = \S$. Because $\frac{1+i}{1-i} \in \S \setminus \{\pm1\}$ for all $i \in {\rm Im}(\ell) \setminus \{0\}$, 
the first inclusion in  
\[ \langle \varphi_{-1},\psi_1 \rangle = \{ \varphi_a, \psi_a\ |\ a = \pm1 \} \subset \{ \varphi_a, \psi_a\ |\ a \in 
\S \} \subset {\rm Aut}(A(\ell,\underline{c})) \] 
is proper.
\\[1ex]
{\it (ii)} If $c_2 \not= 0$, then $\ell^\ast(\underline{c},\underline{c}) = \{ \pm1 \}$ and $\underline{c} \in C_\ell^\mathscr{N} \sqcup C_\ell^\mathscr{K}$ (Proposition \ref{types of admissible triples}).
With this information, Corollary \ref{when all morphisms are constructed} implies that
\[ \langle \varphi_{-1},\psi_1 \rangle = \{ \varphi_a, \psi_a\ |\ a \in \ell^\ast(\underline{c},\underline{c}) \} = {\rm Aut}(A(\ell,\underline{c})), \] 
which completes the proof.
\end{proof}

\subsection{The Kleinian coverings of $\mathscr{C}(k)$ by group actions}

Let $k \subset \ell$ be a quadratic extension in characteristic not 2. In Subsection 4.4 we introduced the full subgroupoid $\mathscr{C}_\ell = \mathscr{C}_\ell(k)$ of 
$\mathscr{C} = \mathscr{C}(k)$. Now we decompose it into two blocks $\mathscr{C}_\ell^0$ and $\mathscr{C}_\ell^1$, defined by their object classes 
$\mathscr{C}_\ell^0 = \{ A \in \mathscr{C}_\ell\ |\ |{\rm Aut}(A)| > 4\}$ and $\mathscr{C}_\ell^1 = \{ A \in \mathscr{C}_\ell\ |\ |{\rm Aut}(A)| = 4\}$.\vspace{0,1cm} Aiming for descriptions 
$Q_\ell^\nu = (C_\ell^\nu,G_\ell^\nu,\gamma_\ell^\nu,\mathscr{F}_\ell^\nu)$ of $\mathscr{C}_\ell^\nu,\ \nu \in \{0,1\}$, we proceed\vspace{0,1cm} to define the data constituting the quadruples $Q_\ell^0$ and 
$Q_\ell^1$.\vspace{0,1cm}

The sets $C_\ell^0 = \{ \underline{c} \in C_\ell\ |\ c_2 = 0 \}$ and $C_\ell^1 = \{ \underline{c} \in C_\ell\ |\ c_2 \not= 0 \}$ are sets of $\ell$-admissible triples, partitioning $C_\ell$.

The groups \mbox{$G_\ell^0 = \ell^\ast >\hspace{-0,2cm}\lhd\ {\rm Gal}(\ell/k)$} and \mbox{$G_\ell^1 = \A^\ast >\hspace{-0,2cm}\lhd\ {\rm Gal}(\ell/k)$} are semi-direct products of the type mentioned in
Subsection 2.2.

The group actions $\gamma_\ell^\nu: G_\ell^\nu \times C_\ell^\nu \to C_\ell^\nu$ are defined uniformly, for both $\nu \in \{0,1\}$, by 
\begin{eqnarray} 
\gamma_\ell^\nu \left( \left( a, \sigma^i \right),\left( c_1, c_2, c_3 \right) \right) = \left( c_1, \frac{c_2}{a^2}, \frac{c_3}{a\overline{a}} \right). 
\end{eqnarray}

The functors $\mathscr{F}_\ell^\nu: \ _{G_\ell^\nu} C_\ell^\nu\ \to\ \mathscr{C}_\ell^\nu$ are defined uniformly,\vspace{0,1cm} for both $\nu \in \{0,1\}$, by 
$\mathscr{F}_\ell^\nu(\underline{c})\ =\ A(\ell, \underline{c})$ on objects $\underline{c}$, and by 
\begin{eqnarray} 
\mathscr{F}_\ell^\nu((a, \sigma^i), \underline{c}, \underline{d})\ =\ \left\{ \begin{array}{ccc} \varphi_a & \mbox{if} & i=0 \\ \psi_a & \mbox{if} & i=1 \end{array} \right. 
\end{eqnarray}
on morphisms $((a, \sigma^i), \underline{c}, \underline{d})$.

It is easily verified that both maps $\gamma_\ell^0$ and $\gamma_\ell^1$ indeed are group actions, and, taking Propositions \ref{constructed algebras} (iv) and \ref{Kleinian subgroup} as well as
Lemma \ref{constructed morphisms} (ii) into account, that both families of maps $\mathscr{F}_\ell^0$ and $\mathscr{F}_\ell^1$ indeed are functors.\vspace{0,1cm}

We also note that, for both $\nu \in \{ 0,1 \}$, the set $C_\ell^\nu$ is a union of $\ell^\ast$-equi\-valence classes in $C_\ell$, and the $\ell^\ast$-equivalence classes of $C_\ell^\nu$ coincide with 
the $G_\ell^\nu$-orbits of $C_\ell^\nu$. Thus, we are about to enhance our earlier approach with categorical and functorial structure.

\begin{pro} \label{local description}
For every quadratic extension $k \subset \ell$ in characteristic \mbox{not 2,} the following holds true. 
\\[1ex]
(i) $\mathscr{C}_\ell = \mathscr{C}_\ell^0 \amalg \mathscr{C}_\ell^1$.
\\[1ex]
(ii) $Q_\ell^0 = (C_\ell^0,G_\ell^0,\gamma_\ell^0,\mathscr{F}_\ell^0)$ is a description of $\mathscr{C}_\ell^0$ by a group action. 
\\[1ex]
(iii) $Q_\ell^1 = (C_\ell^1,G_\ell^1,\gamma_\ell^1,\mathscr{F}_\ell^1)$ is a full description of $\mathscr{C}_\ell^1$ by a group action.
\end{pro}

\begin{proof}
{\it (i)} The claimed decomposition holds by definition of $\mathscr{C}_\ell^0$ and $\mathscr{C}_\ell^1$.
\\[1ex]
{\it (ii)-(iii)} Both functors $\mathscr{F}_\ell^0$ and $\mathscr{F}_\ell^1$ are dense by Theorem \ref{reduction of objects} (ii) and Proposition \ref{Kleinian subgroup}, faithful by definition, and quasi-full 
by Theorem \ref{reduction of morphisms}. Thus\linebreak $Q_\ell^0$ is a description of $\mathscr{C}_\ell^0$ by a group action, and $Q_\ell^1$ is a description \mbox{of $\mathscr{C}_\ell^1$} by a group action. 
In order to show that the functor $\mathscr{F}_\ell^1$ even is full, let $(\underline{c},\underline{d}) \in C_\ell^1 \times C_\ell^1$ and a morphism 
$\mu: A(\ell,\underline{c}) \to A(\ell,\underline{d})$ in $\mathscr{C}_\ell^1$\vspace{0,1cm} be given. Then $(\underline{c},\underline{d}) \in \left( C_\ell^\mathscr{N} \times C_\ell^\mathscr{N} \right) \sqcup 
\left( C_\ell^\mathscr{K} \times C_\ell^\mathscr{K} \right)$ by Proposition \ref{types of admissible triples},\vspace{0,1cm} whence Corollary \ref{when all morphisms are constructed} yields a morphism 
$((a,\sigma^i),\underline{c},\underline{d}): \underline{c} \to \underline{d}$ in $\ _{G_\ell^1} C_\ell^1$, such that $\mathscr{F}_\ell^1((a,\sigma^i),\underline{c},\underline{d}) = \mu$.
\end{proof}

\noindent
At last, the Kleinian covering of $\mathscr{C}(k)$ by group actions enters the stage. Again, we write $\mathscr{C} = \mathscr{C}(k),\ \mathscr{Q} = \mathscr{Q}(k)$ etc., for brevity.

\begin{cor} \label{Kleinian covering by group actions}
Let $k$ be a field of characteristic not 2, and let $\mathscr{L}$ be a transversal for $\mathscr{Q}/\hspace{-0,1cm}\simeq$. Then the family 
\[ Q = \left( Q_\ell^\nu \right)_{(\ell,\nu) \in \mathscr{L} \times \{0,1\}} = (C_\ell^\nu,G_\ell^\nu,\gamma_\ell^\nu,\mathscr{F}_\ell^\nu)_{(\ell,\nu) \in \mathscr{L} \times \{0,1\}} \]
of descriptions $Q_\ell^\nu$ of $ \mathscr{C}_\ell^\nu$ is a covering of $\mathscr{C}$ by group actions.
\end{cor}

\begin{proof}
For each $(\ell,\nu) \in \mathscr{L} \times \{0,1\}$, the quadruple $Q_\ell^\nu$ is a description of $\mathscr{C}_\ell^\nu$ by a group action (Proposition \ref{local description} (ii)-(iii)). 
Furthermore, the full subgroupoids $\mathscr{C}_\ell^\nu \subset \mathscr{C}$ cover $\mathscr{C}$, as the identities of object classes 
\[ \mathscr{C} = \bigcup_{\ell \in \mathscr{L}} \mathscr{C}_\ell = \bigcup_{\ell \in \mathscr{L}} \left(\mathscr{C}_\ell^0 \sqcup \mathscr{C}_\ell^1 \right) = 
\bigcup_{(\ell,\nu) \in \mathscr{L} \times \{0,1\}} \mathscr{C}_\ell^\nu \]
hold by Proposition \ref{L-coverings} and Proposition \ref{local description} (i)).
\end{proof}

\noindent
We call the family $Q$, displayed in Corollary \ref{Kleinian covering by group actions}, the {\it Kleinian covering of $\mathscr{C}$ by group actions}. It admits a noteworthy refinement, which evolves from 
the decomposition of $\mathscr{C}_\ell^0,\ \ell \in \mathscr{L}$, into the blocks\vspace{0,1cm} $\mathscr{C}_\ell^2 = \mathscr{N}_\ell \cap \mathscr{C}_\ell^0$ and $\mathscr{C}_\ell^3 = \mathscr{S}_\ell$, 
via the following descriptions $Q_\ell^\nu = (C_\ell^\nu,G_\ell^\nu,\gamma_\ell^\nu,\mathscr{F}_\ell^\nu)$ of\vspace{0,1cm} $\mathscr{C}_\ell^\nu,\ \nu \in \{2,3\}$. The sets $C_\ell^2 = C_\ell^\mathscr{N} \cap C_\ell^0$ 
and $C_\ell^3 = C_\ell^\mathscr{S}$ are subsets of $C_\ell^0$,\vspace{0,1cm} partitioning $C_\ell^0$ by Proposition \ref{types of admissible triples}. The groups $G_\ell^2$ and $G_\ell^3$,\vspace{0,1cm} acting 
on these sets, are \mbox{$G_\ell^2 = G_\ell^3 = G_\ell^0 = \ell^\ast >\hspace{-0,2cm}\lhd\ {\rm Gal}(\ell/k)$.} The group actions\vspace{0,1cm} $\gamma_\ell^\nu: G_\ell^\nu \times C_\ell^\nu \to C_\ell^\nu,$ 
\mbox{$\nu \in \{2,3\}$,} are both induced from $\gamma_\ell^0$, i.e.~they satisfy formula (5).\vspace{0,1cm} The func\-tors $\mathscr{F}_\ell^\nu: \ _{G_\ell^\nu} C_\ell^\nu\ \to\ \mathscr{C}_\ell^\nu,\ 
\nu \in \{2,3\}$, are both induced from $\mathscr{F}_\ell^0$,\vspace{0,1cm} i.e.~they satisfy $\mathscr{F}_\ell^\nu(\underline{c})\ =\ A(\ell, \underline{c})$ on objects, and formula (6) on morphisms.

\begin{pro} \label{refined local description}
For every quadratic extension $k \subset \ell$ in characteristic \mbox{not 2,} the following holds true. 
\\[1ex]
(i) $\mathscr{C}_\ell^0 = \mathscr{C}_\ell^2 \amalg \mathscr{C}_\ell^3$.
\\[1ex]
(ii) $Q_\ell^2 = (C_\ell^2,G_\ell^2,\gamma_\ell^2,\mathscr{F}_\ell^2)$ is a full description of $\mathscr{C}_\ell^2$ by a group action. 
\\[1ex]
(iii) $Q_\ell^3 = (C_\ell^3,G_\ell^3,\gamma_\ell^3,\mathscr{F}_\ell^3)$ is a description of $\mathscr{C}_\ell^3$ by a group action.
\end{pro}

\begin{proof}
{\it (i)} The decomposition $\mathscr{C}_\ell = \mathscr{N}_\ell \amalg \mathscr{S}_\ell \amalg \mathscr{K}_\ell$, valid by Proposition \ref{L-coverings}, induces the decomposition 
\[ \mathscr{C}_\ell^0 = \left( \mathscr{N}_\ell \cap \mathscr{C}_\ell^0 \right ) \amalg \left( \mathscr{S}_\ell \cap \mathscr{C}_\ell^0 \right) \amalg \left( \mathscr{K}_\ell \cap \mathscr{C}_\ell^0 \right) =  
\left( \mathscr{N}_\ell \cap \mathscr{C}_\ell^0 \right ) \amalg \mathscr{S}_\ell = \mathscr{C}_\ell^2 \amalg \mathscr{C}_\ell^3, \]
where the second identity stems from $\mathscr{S}_\ell \subset \mathscr{C}_\ell^0$ and $\mathscr{K}_\ell \cap \mathscr{C}_\ell^0 = \emptyset$, valid by Theorem \ref{reduction of objects} (ii), 
Proposition \ref{types of admissible triples} (ii)-(iii), and Proposition \ref{Kleinian subgroup}.
\\[1ex]
{\it (ii)-(iii)} Arguing as in the proof of Proposition \ref{local description} (ii)-(iii), one finds that $Q_\ell^2$ is a description of $\mathscr{C}_\ell^2$ by a group action, and $Q_\ell^3$ is a 
description $\mathscr{C}_\ell^3$ by a group action. Observing that $C_\ell^2 \subset C_\ell^\mathscr{N}$, even fullness of the functor $\mathscr{F}_\ell^2$ is proved by the same arguments which proved
fullness of $\mathscr{F}_\ell^1$, above.
\end{proof}

\begin{cor} \label{refined Kleinian covering by group actions}
Let $k$ be a field of characteristic not 2, and let $\mathscr{L}$ be a transversal for $\mathscr{Q}/\hspace{-0,1cm}\simeq$. Then the family 
\[ Q_{\rm ref} = \left( Q_\ell^\nu \right)_{(\ell,\nu) \in \mathscr{L} \times \underline{3}} = (C_\ell^\nu,G_\ell^\nu,\gamma_\ell^\nu,\mathscr{F}_\ell^\nu)_{(\ell,\nu) \in \mathscr{L} \times \underline{3}} \]
of descriptions $Q_\ell^\nu$ of $ \mathscr{C}_\ell^\nu$ is a covering of $\mathscr{C}$ by group actions.
\end{cor}

\begin{proof}
For each $(\ell,\nu) \in \mathscr{L} \times \underline{3}$, the quadruple $Q_\ell^\nu$ is a description of $\mathscr{C}_\ell^\nu$ by a group action (Proposition \ref{local description} (iii) and 
Proposition \ref{refined local description} (ii)-(iii)). Moreover, the full subgroupoids $\mathscr{C}_\ell^\nu \subset \mathscr{C}$ cover $\mathscr{C}$, as the identities of object classes 
\[ \mathscr{C} = \bigcup_{\ell \in \mathscr{L}} \mathscr{C}_\ell = \bigcup_{\ell \in \mathscr{L}} \left(\mathscr{C}_\ell^1 \sqcup \mathscr{C}_\ell^2  \sqcup \mathscr{C}_\ell^3 \right) = 
\bigcup_{(\ell,\nu) \in \mathscr{L} \times \underline{3}} \mathscr{C}_\ell^\nu \]
hold by Proposition \ref{L-coverings}, Proposition \ref{local description} (i) and Proposition \ref{refined local description} (i).
\end{proof}

\noindent
We call the family $Q_{\rm ref}$, appearing in Corollary \ref{refined Kleinian covering by group actions}, the {\it refined Kleinian covering of $\mathscr{C}$ by group actions}. It enables us to 
display the automorphism groups, up to isomorphism, of all objects in $\mathscr{C} = \mathscr{C}(k)$. For that purpose, we decompose $\mathscr{N} = \mathscr{N}(k)$ into two blocks 
$\mathscr{N}^0$ and $\mathscr{N}^1$, defined by their object classes $\mathscr{N}^0 = \{ A \in \mathscr{N}\ |\ |{\rm Aut}(A)| > 4 \}$ and $\mathscr{N}^1 = \{ A \in \mathscr{N}\ |\ |{\rm Aut}(A)| = 4 \}$.

\begin{cor} \label{automorphism groups}
Let $k$ be a field of characteristic not 2, and let $\mathscr{L}$ be a transversal for $\mathscr{Q}/\hspace{-0,1cm}\simeq$. Then the following holds true.
\\[1ex]
(i)\hspace{0,2cm} $\mathscr{C} = \mathscr{N}^0 \amalg \mathscr{N}^1 \amalg \mathscr{S} \amalg \mathscr{K}$.
\\[1ex]
(ii) If $A \in \mathscr{N}^0$, then $N_r(A)\ \tilde{\to}\ \ell$ for a unique $\ell \in \mathscr{L}$, and 
\[ {\rm Aut}(A)\ \tilde{\to}\ \S(\ell/k) >\hspace{-0,2cm}\lhd\ {\rm C}_2. \]
(iii) If $A \in \mathscr{N}^1$, then ${\rm Aut}(A)$ is Klein's four-group.
\\[1ex]
(iv) If $A \in \mathscr{S}$, then ${\rm Aut}(A)\ \tilde{\to}\ A^\ast/k^\ast$.
\\[1ex]
(v) If $A \in \mathscr{K}$, then ${\rm Aut}(A)$ is Klein's four-group.
\end{cor}

\begin{proof}
{\it (i)} The decomposition $\mathscr{N} = \mathscr{N}^0 \amalg \mathscr{N}^1$ holds by definition of $\mathscr{N}^0$ and $\mathscr{N}^1$. Together with Proposition \ref{decomposition}, it yields the claimed 
decomposition \mbox{of $\mathscr{C}$.} 
\\[1ex]
{\it (ii)} If $A \in \mathscr{N}^0$, then $N_r(A)$ is the unique Kleinian subfield of $A$, by Theorem \ref{reduction of objects} (i). In particular, $k \subset N_r(A)$ is a quadratic extension, 
so\vspace{0,1cm} $N_r(A)\ \tilde{\to}\ \ell$ holds for a unique $\ell \in \mathscr{L}$. Consequently,\vspace{0,1cm} 
$A \in \mathscr{N}_\ell \cap \mathscr{N}^0 = \mathscr{N}_\ell \cap \mathscr{C}_\ell^0 = \mathscr{C}_\ell^2$. The functor $\mathscr{F}_\ell^2: \ _{G_\ell^2}C_\ell^2\ \to\ \mathscr{C}_\ell^2$ is an equivalence of 
categories, by Proposition \ref{refined local description} (ii). Therefore 
$A\ \tilde{\to}\ A(\ell,\underline{c})$\vspace{0,1cm} for some $\underline{c} \in C_\ell^2 = C_\ell^\mathscr{N} \cap C_\ell^0$, and 
${\rm Aut}(A)\ \tilde{\to}\ {\rm Aut}(A(\ell,\underline{c}))\ \tilde{\to}\ {\rm Mor}_{_{G_\ell^2}C_\ell^2}(\underline{c},\underline{c})$. By definition of the group action groupoid $_{G_\ell^2}C_\ell^2$, and 
because $c_2 = 0$ and $c_3 \not= 0$, the latter\vspace{0,1cm} automorphism group admits the following sequence of canonical isomorphisms and identities:
\begin{eqnarray*} 
{\rm Mor}_{_{G_\ell^2}C_\ell^2}(\underline{c},\underline{c}) & \tilde{\to} & \{ (a,\sigma^i) \in \ell^\ast >\hspace{-0,2cm}\lhd\ {\rm Gal}(\ell/k)\ |\ \gamma_\ell^2((a,\sigma^i),\underline{c}) = \underline{c} \} 
\\
 & = & \{ (a,\sigma^i) \in \ell^\ast >\hspace{-0,2cm}\lhd\ {\rm Gal}(\ell/k)\ |\ a \in \ell^\ast(\underline{c},\underline{c}) \} 
\\[1ex]
 & = & \{ (a,\sigma^i) \in \ell^\ast >\hspace{-0,2cm}\lhd\ {\rm Gal}(\ell/k)\ |\ a \in \S(\ell/k) \} 
\\[1ex]
 & = & \S(\ell/k) >\hspace{-0,2cm}\lhd\ {\rm Gal}(\ell/k)
\\[1ex]
 & \tilde{\to} & \S(\ell/k) >\hspace{-0,2cm}\lhd\ {\rm C}_2.
\end{eqnarray*}
{\it (iii)} If $A \in \mathscr{N}^1$, then ${\rm Aut}(A)$ contains Klein's four-group as a subgroup and $|{\rm Aut}(A)| = 4$, whence the statement  follows.
\\[1ex]
{\it (iv)} If $A \in \mathscr{S}$, then $A$ is a central skew field over $k$. Therefore Skolem-Noether Theorem \cite[Theorem 7.21]{Re75} applies to $A$, whence the claimed isomorphism follows.
\\[1ex]
{\it (v)} If $A \in \mathscr{K}$, then $A$ is a Galois extension of $k$ whose Galois group is Klein's four-group (Proposition \ref{decomposition}), and ${\rm Aut}(A) = {\rm Gal}(A/k)$.
\end{proof}

\subsection{Epilogue}

In retrospect, we see that the ``{\it construction and reduction}'' approach, presented in Sections 3 and 4, provides a sufficiently solid basis for a systematic treatment of the {\it classification problem} 
of $\mathscr{C}$. The information it conveys about morphisms in $\mathscr{C}$ is however minimal: it just gives a {\it criterion for the existence} of an isomorphism 
$A(\ell,\underline{c})\ \tilde{\to}\ A(\ell,\underline{d})$ in terms of $\underline{c}$ and $\underline{d}$ (Theorem \ref{reduction of morphisms}). As soon as one changes the perspective and tries to 
{\it understand the structure} of the groupoid $\mathscr{C}$, more information about morphisms is required. The refined Kleinian covering of $\mathscr{C}$ meets this need. It supplements the construction 
and reduction approach by ``{\it local descriptions}'' of $\mathscr{C}$, i.e.~descriptions of the full subgroupoids $\mathscr{C}_\ell^\nu$ by group actions, such that the object class $\mathscr{C}$ is 
covered by the object classes $\mathscr{C}_\ell^\nu$, where $(\ell,\nu) \in \mathscr{L} \times \underline{3}$. The proof of \mbox{Corollary \ref{automorphism groups} (ii)} illustrates the usefulness of 
this deepened insight. 

The Kleinian coverings of $\mathscr{C}$ also exemplify a {\it phenomenon} which appears to be ubiquitous although far from evident, namely {\it that a given groupoid $\mathscr{Y}$ admits an explicitly 
and constructively presented covering by group actions at all.} Apart from the groupoid $\mathscr{C}$, the occurrence of this phenomenon has in fact been observed in various other contexts. As a sample, 
let us mention the groupoid of all 2-dimensional real division algebras \cite[Theorem 3.3]{Di05}, the groupoid of all 4-dimensional absolute valued algebras \cite[Proposition 5.3]{Di12}, and the groupoid of 
all 8-dimensional composition algebras over a field of characteristic not 2 \cite[Corollary 4.6]{Al17}.

A general theory however, that would be able to outline and explain the scope of validity of this phenomenon, seems not to exist at present. Yet whenever it does occur for a given 
groupoid $\mathscr{Y}$, it provides a holistic picture of $\mathscr{Y}$ in the sense that the classification problem of $\mathscr{Y}$ locally can be viewed as the normal form problem of a group 
action, and, whenever a local description of $\mathscr{Y}$ is full, then the described full subgroupoid is equivalent to an explicitly presented group action groupoid.

The author of the present article enjoyed the privilege to visit the Mathe\-matics Department of Z\"{u}rich University during the years 1987-1992 as a member of the Algebra group, which then was led by
Peter Gabriel. In numerous private discussions, fine and memorable ones, Gabriel repeatedly stressed his viewpoint of the primacy of the concept of a group action. He even based his last lectures in 
Linear Algebra, documented in \cite{Ga96}, on that concept. The present article is a late outgrowth of the influental ideas that shaped the author during his stay at Z\"{u}rich University. It confirms the
conjecture that classification problems, in great generality, can be understood as normal form problems of group actions.

\vspace{1cm}
\noindent 
\begin{tabular}{@{}rl}
Corresponding author: & Ernst Dieterich
\\[1ex]
Address: & Department of Mathematics\\
 & Uppsala University\\
 & Box 480\\
 & SE-751 06 Uppsala, Sweden
\\[1ex]
e-mail: & Ernst.Dieterich@math.uu.se
\\[1ex]
Tel: & +46-18-4717622
\\[1ex]
Fax: & +46-18-4713201
\end{tabular}
\end{document}